\documentclass{article}
\usepackage{graphicx,amsmath,amsthm,amscd,amssymb,pifont}

\newtheorem{thm}{Theorem}[section]
\newtheorem{cor}[thm]{Corollary}
\newtheorem{lem}[thm]{Lemma}
\newtheorem{prop}[thm]{Proposition}
\newtheorem*{main thm}{Main Theorem}
\newtheorem*{main lem}{Main Lemma}
\theoremstyle{definition}
\newtheorem{defn}[thm]{Definition}
\newtheorem{notn}[thm]{Notation}
\newtheorem{example}[thm]{Example}

\theoremstyle{remark}
\newtheorem{rem}[thm]{Remark}
\numberwithin{equation}{section}

\renewcommand{\Im}{\textrm{Im}}

\newcommand{\R}{\mathbb R}
\newcommand{\C}{\mathbb C} 
\newcommand{\Z}{\mathbb Z}
\newcommand{\N}{\mathbb N}
 
\newcommand{\Lie}{\mathcal{L}}

\newcommand{\tens}{\otimes} 
\newcommand{\dsum}{\oplus}

\newcommand{\iso}{\simeq}

\renewcommand{\d}{\textrm{d}}
\renewcommand{\phi}{\varphi} 
\renewcommand{\to}{\longrightarrow}
\newcommand{\oto}[1]{\overset{#1}\to}

\renewcommand{\mapsto}{\longmapsto}

\newcommand{\comp}{\circ}
\newcommand{\sr}{\mathcal}

\newcommand{\Aut}{\textrm{Aut}}

\newcommand{\del}{\partial}

\newcommand{\fl}[1]{\left\lfloor #1 \right\rfloor}

\renewcommand{\^}{\wedge} 
 
\renewcommand{\epsilon}{\varepsilon}
\newcommand{\hide}[1]{}

\newcommand{\Id}{\textrm{Id}}
\newcommand{\LP}{\mathfrak{L}}

\begin{document}

\title{Local classification of generalized complex structures}
\author{Michael Bailey}

\maketitle

\begin{abstract}
We give a local classification of generalized complex structures.  About a point, a generalized complex structure is equivalent to a product of a symplectic manifold with a holomorphic Poisson manifold.  We use a Nash-Moser type argument in the style of Conn's linearization theorem.
\end{abstract}

\tableofcontents

\section{Introduction}

Generalized complex geometry is a generalization of both symplectic and complex geometry, introduced by Hitchin \cite{Hitchin2003}, and developed by Gualtieri (we refer to his recent publication \cite{Gualtieri2011} rather than his thesis), Cavalcanti \cite{Cavalcanti} and by now many others.  Its applications include the study of 2-dimensional supersymmetric quantum field theories, which occur in topological string theory, as well as compactifications of string theory with fluxes, and mirror symmetry.  Or alternatively, it may be motivated by other geometries---for example, bi-Hermitian geometry, now realized as the generalized complex analogue of K\"ahler geometry.

\begin{defn}
A generalized complex structure on a manifold $M$ is a complex structure,
$$J : TM \dsum T^*M \to TM \dsum T^*M\qquad (J^2 = -\Id),$$
on the ``generalized tangent bundle'' ${TM \dsum T^*M}$, which is orthogonal with respect to the standard symmetric pairing between $TM$ and $T^*M$, and whose $+i$-eigenbundle is involutive with respect to the Courant bracket, defined as follows: let $X,Y \in \Gamma(TM)$ be vector fields and $\xi,\eta \in \Gamma(T^*M)$ be 1-forms; then
\begin{equation}\label{bracket formula}
[X+\xi,Y+\eta] = [X,Y]_{\textnormal{Lie}} + \Lie_X \eta - \iota_Y d\xi.
\end{equation}
\end{defn}

The Courant bracket (actually, in this form, due to Dorfman \cite{Dorfman}) usually has an additional twisting term involving a closed 3-form.  However, every such bracket is in a certain sense locally equivalent to the untwisted bracket above, and since this paper studies the local structure, without loss of generality we ignore the twisting.

\begin{example}\label{symplectic structure}
If $\omega:TM \to T^*M$ is a symplectic structure, then
$$J_\omega =
\left[\begin{array}{cc}
0 & -\omega^{-1} \\
\omega & 0
\end{array}\right]$$
is a generalized complex structure.
\end{example}

\begin{example}\label{complex structure}
If $I:TM \to TM$ is a complex structure, then
$$J_I = 
\left[\begin{array}{cc}
-I & 0 \\
0 & I^*
\end{array}\right]$$
is a generalized complex structure.
\end{example}

\begin{rem}
A generalized complex structure may be of complex type or symplectic type \emph{at a point $p$}, if it is of one of the above forms on the vector space $T_pM\dsum T^*_pM$; however, its type may vary from one point to another.
\end{rem}

\begin{example}
If $J_1$ is a generalized complex structure on $M_1$ and $J_2$ is a generalized complex structure on $M_2$, then $J_1 \times J_2$ is a generalized complex structure on $M_1 \times M_2$ in the obvious way.
\end{example}

\begin{defn}
A \emph{generalized diffeomorphism} (or \emph{Courant isomorphism}) ${\Phi : TM\dsum T^*M \to TN\dsum T^*N}$ is a map of $TM\dsum T^*M$ to $TN \dsum T^*N$, covering some diffeomorphism, which respects the Courant bracket, the symmetric pairing, and the projection to the tangent bundle.
\end{defn}

The first result on the local structure of generalized complex structures was due to Gualtieri \cite{Gualtieri2011}.  It was strengthened by Abouzaid and Boyarchenko \cite{AbouzaidBoyarchenko}, as follows:

\begin{thm}[Abouzaid, Boyarchenko]
If $M$ is a generalized complex manifold and $p\in M$, then there is a neighourhood of $p$ which isCourant isomorphic to a product of a generalized complex manifold of symplectic type everywhere with a generalized complex manifold which is of complex type at the image of $p$.
\end{thm}

This resembles Weinstein's local structure theorem for Poisson structures \cite{Weinstein}, where in this case the point of complex type is analogous to the point where the Poisson rank is zero.  In fact, a generalized complex structure induces a Poisson structure, for which this result produces the Weinstein decomposition.

Thus, the remaining question in the local classification of generalized complex structures is: what kinds of generalized complex structures occur near a point of complex type?  There is a way in which any holomorphic Poisson structure (see Section \ref{holomorphic Poisson} below) induces a generalized complex structure (as described in Section \ref{deformation section}).  Our main result, then, is as follows:

\begin{main thm}
Let $J$ be a generalized complex structure on a manifold $M$ which is of complex type at point $p$.  Then, in a neighbourhood of $p$, $J$ is Courant-equivalent to a generalized complex structure induced by a holomorphic Poisson structure, for some complex structure near $p$.
\end{main thm}

This is finally proven in Section \ref{lemma implies theorem}.  Most of the work happens in earlier sections, in proving the following lemma:

\begin{main lem}
Let $J$ be a generalized complex structure on the closed unit ball $B_1$ about the origin in $\C^n$.  Suppose $J$ is a small enough deformation of the complex structure on $B_1$, and suppose that $J$ is of complex type at the origin.  Then, in a neighbourhood of the origin, $J$ is Courant-equivalent to a deformation of the complex structure by a holomorphic Poisson structure.
\end{main lem}

In Section \ref{deformation section}, we explain how one generalized complex structure may be understood as a deformation of another.  By the smallness condition we mean that there is some $k\in\N$ such that if the deformation is small enough in its $C^k$-norm then the conclusion holds.  (See Section \ref{norm section} for our conventions for $C^k$-norms.)  The proof of the Main Lemma is in Section \ref{prove main lemma} (modulo technical results in Sections \ref{verifying SCI} and \ref{verifying hypotheses}).

In some sense, then, generalized complex structures are holomorphic Poisson structures twisted by (possibly) non-holomorphic Courant gluings.  A generalized complex manifold may not, in general, admit a global complex structure \cite{CavalcantiGualtieri2007} \cite{CavalcantiGualtieri}, emphasizing the local nature of our result.

\begin{rem}
We stress that the Main Theorem only tells us that there is \emph{some} complex structure near $p$, and a holomorphic Poisson structure with respect to it, which generate a given generalized complex structure.  Neither of these data are unique---if they were, then in most cases we would be able to assign a consistent global complex structure to the manifold, in contradiction with the remark above.  We will address some of these issues in forthcoming work.
\end{rem}

\begin{rem}
Throughout this paper, we assume for simplicity that all functions and sections are $C^\infty$-smooth, though in fact this is not a major point---the arguments carry through for finite smoothness class, and then the resulting holomorphicity implies that, in local complex coordinates, the structure is in fact $C^\infty$.
\end{rem}

\subsection{Holomorphic Poisson structures}\label{holomorphic Poisson}

A holomorphic Poisson structure on a complex manifold $M$ is given by a holomorphic bivector field $\beta \in \Gamma(\^ ^2 T_{1,0}M),\; \bar\del\beta=0,$ for which the Schouten bracket, $[\beta,\beta]$, vanishes.  $\beta$ determines a Poisson bracket on holomorphic functions, $\{f,g\} = \beta(df,dg)$.  The holomorphicity condition, $\bar\del\beta=0$, means that, if $\beta$ is written in local coordinates,
$$\beta = \sum_{i,j} \beta_{ij} \frac{d}{dz_i}\^\frac{d}{dz_j},$$
then the component functions, $\beta_{ij}$, are holomorphic.  For a review of holomorphic Poisson structures, see, for example, \cite{LSX}.

The \emph{type-change locus} of a generalized complex structure induced by a holomorphic Poisson structure, that is, the locus where the Poisson rank changes, is determined by the vanishing of an algebraic function of the component functions above, thus,
\begin{cor}
The type-change locus of a generalized complex structure locally admits the structure of an analytic subvariety.
\end{cor}

(The global analyticity of the type change locus will be addressed in upcoming work.)  So as a consequence of the Main Theorem, the local structure is not too badly behaved---certainly better than generic smooth Poisson structures---but how much we can say about the local structure of generalized complex manifolds now depends on what we can say about the local structure of holomorphic Poisson structures, which is less than we might hope.  In particular, we believe the following to be an open question:
\begin{itemize}
\item Is every holomorphic Poisson structure locally equivalent to one which is polynomial in some coordinates?
\end{itemize}
We are unaware of any counterexamples, and there are some partial results \cite{DufourWade} \cite{Lohrmann}.

\subsection{Outline of the proof of the Main Lemma}\label{outline of proof}

In Section \ref{deformation section} we describe the deformation complex for generalized complex structures, and how it interacts with \emph{generalized flows} coming from \emph{generalized vector fields}.  In Section \ref{infinitesimal case} we solve an infinitesimal version of the problem, by showing that, to first order, an infinitesimal generalized complex deformation of a holomorphic Poisson structure is equivalent to another holomorphic Poisson structure.  This is just a cohomological calculation.  Then the full problem is solved by iterating an approximate version of the infinitesimal solution:

\subsubsection*{The iteration}

At each stage of the iteration, we have a generalized complex structure which is a deformation of a given complex structure.  We seek to cancel the part of this deformation which is \emph{not} a Poisson bivector.  We construct a generalized vector field whose generalized flow acting on the deformation should cancel this non-bivector part, to first order.  Then after each stage the unwanted part of the deformation should shrink quadratically.  We mention two problems with this algorithm:

\subsubsection*{Loss of derivatives}

Firstly, at each stage we ``lose derivatives,'' meaning that the $C^k$-convergence will depend on ever higher $C^{k+i}$-norms.  The solution is to apply Nash's smoothing operators at each stage to the generalized vector field, where the smoothing degree is carefully chosen to compensate for loss of derivatives while still achieving convergence.  A good general reference for this sort of technique (in the context of compact manifolds) is \cite{Hamilton}, and it is tempting to try to apply the Nash-Moser implicit function theorem directly. However, this is frustrated by the second issue:

\subsubsection*{Shrinking neighbourhoods}

Since we are working on a neighbourhood of a point $p$, the generalized vector field will not integrate to a generalized diffeomorphism of the whole neighbourhood.  Thus, after each stage we may have to restrict our attention to a smaller neighbourhood of $p$.  If the radius restriction at each stage happens in a controlled way, then the limit will be defined on a ball of radius greater than 0.  However, we still must have ways to cope with an iteration, not over a single space of sections, but over many spaces of sections, one for each neighbourhood.

The technique for proving Nash-Moser type convergence results on shrinking neighbourhoods comes from Conn \cite{Conn}.  In Section \ref{SCI section} we describe a relatively recent formalization of this technique, by Miranda, Monnier and Zung \cite{MirandaMonnierZung} \cite{MonnierZung}.  In fact, much of the heavy lifting is done by a general technical lemma of theirs (Theorem \ref{Miranda-Monnier-Zung}), which we have somewhat generalized (to weaken the estimates, and account for nonlinear actions).  Even so, we must prove estimates for the behaviour of generalized flows acting on generalized deformations (in Sections \ref{verifying SCI} and \ref{verifying hypotheses}).

\subsection*{Reflections on the method}
We hold out hope that there is an easier proof.  We expect that the method we have used is in some sense the most direct, but, as one can see in Sections \ref{verifying SCI} and \ref{verifying hypotheses}, a lot of effort must be exerted to prove estimates.

In trying ``softer'' methods, \emph{\`a la} Moser or other tricks, we encountered obstacles.  This is not surprising, since the full Newlander-Nirenberg theorem comes out as a corollary, so the proof should be at least as hard, unless it were possible to use the N.-N. in some essential way.  (This seems doubtful.)

There are by now a variety of proofs of the Newlander-Nirenberg theorem which might serve as inspiration for different proofs of our theorem; however, the context of our theorem seems different enough that would not be straightforward.

\subsection{Acknowledgements}

This paper is drawn from Chapter 2 of my Ph.D. thesis \cite{Bailey}.  During my time in graduate school at the University of Toronto, I received support, both material and mathematical, from many sources.  In particular I would like to thank my advisors, Marco Gualtieri and Yael Karshon, as well as Brent Pym, Jordan Watts, Jordan Bell and Ida Bulat.

\section{The deformation complex and generalized flows}\label{deformation section}
We will now describe the deformation complex for generalized complex structures.  We make use of the fact that a generalized complex structure is determined by its $+i$-eigenbundle.  Except where we remark otherwise, the results in this section can be found in \cite[Section 5]{Gualtieri2011}.

If $V$ is a vector bundle, let $\Gamma(V)$ denote its smooth sections, and let $V_\C = \C \tens V$.  Let $T$ be the tangent bundle of some manifold $M$.  Let $L \subset T_\C \dsum T_\C^*$ be the $+i$-eigenbundle for an ``initial'' generalized complex structure.  We will often take this initial structure to be induced by a complex structure as in Example \ref{complex structure}, in which case
$$L = T_{0,1} \dsum T^*_{1,0}.$$
Another example arises from a holomorphic Poisson structure.  If ${\beta : T^*_{1,0} \to T_{1,0}}$ is a holomorphic Poisson bivector (on a manifold with a given underlying complex structure $I$), then we define the corresponding generalized complex structure
$$J_\beta = 
\left[\begin{array}{cc}
-I & \Im(\beta) \\
0 & I^*
\end{array}\right],$$
with $+i$-eigenbundle
\begin{equation}\label{L for holomorphic bivector}
L = T_{0,1} \dsum \textnormal{graph}(\beta).
\end{equation}

In any case, the $+i$-eigenbundle of a generalized complex structure is a maximal isotropic subbundle of ``real rank zero.''  Using the symmetric pairing, we choose an embedding of $L^*$ in $T_\C \dsum T_\C^*$, which will be transverse to $L$ and isotropic.  One such choice is $L^* \iso \bar{L}$ (that this choice works is just the meaning of ``real rank zero'' in this case), though we may take others.  Any maximal isotropic $L_\epsilon$ close to $L$ may thus be realized as
\begin{equation}\label{define deformation}
L_\epsilon = (1 + \epsilon)L,
\end{equation}
where $\epsilon : L \to L^* \subset T_\C \dsum T^*_\C$.  As a consequence of the maximal isotropic condition on $L_\epsilon$, $\epsilon$ will be antisymmetric, and we can say that $\epsilon \in \Gamma(\^ ^2 L^*)$.  In fact, for any $\epsilon \in \Gamma(\^ ^2 L^*)$, $L_\epsilon$ is the $+i$--eigenbundle of an almost generalized complex structure.  Of course, for $L_\epsilon$ to be integrable, $\epsilon$ must satisfy a differential condition---the Maurer-Cartan equation (see Section \ref{Maurer-Cartan section}).

\subsubsection{$L^*$ convention}\label{L* convention}
When the initial structure is complex, we will use the convention $L^* \iso \bar{L}$, so that $L^* = T_{1,0} \dsum T^*_{0,1}$.  Since the only requirement on the embedding of $L^*$ is that it be transverse to $L$ and isotropic (and thus give a representation of $L^*$ by the pairing), we will take this same choice of $L^*$ whenever possible; that is, we henceforth fix the notation
\begin{equation}
L^* = T_{1,0} \dsum T^*_{0,1},
\end{equation}
regardless of which eigenbundle $L$ we are dealing with.

\begin{rem}
If the initial structure is complex and $\epsilon = \beta \in \Gamma(\^ ^2 T_{1,0})$ is a holomorphic Poisson bivector, then the deformed eigenbundle $L_\epsilon$ agrees with \eqref{L for holomorphic bivector}.
\end{rem}

\subsection{Generalized Schouten bracket and Lie bialgebroid structure}
A \emph{Lie algebroid} is a vector bundle $A \to M$ whose space of sections has the structure of a Lie algebra, along with an \emph{anchor map} ${\rho : A \to TM}$, such that the Leibniz rule,
$$[X,fY] = \left(\rho(X)\cdot f\right)\, Y + f[X,Y],$$
holds (for $X,Y \in \Gamma(A)$ and $f \in \C^\infty(M)$).

While $T\dsum T^*$ is not a Lie algebroid for the Courant bracket (which isn't antisymmetric), the restriction of the bracket to the maximal isotropic $L$ \emph{does} give a (complex) Lie algebroid structure.  From this, there is a naturally-defined differential
$$d_L : \Gamma(\^ ^k L^*) \to \Gamma(\^ ^{k+1} L^*)$$
as well as an extension of the bracket (in the manner of Schouten) to higher wedge powers of $L$.  But $L^*$ is also a Lie algebroid and together they form a Lie bialgebroid (actually, a differential Gerstenhaber algebra if we consider the wedge product), meaning that $d_L$ is a derivation for the bracket on $\^ ^\bullet L^*$:
\begin{equation}\label{bialgebroid}
d_L[\alpha,\beta] = [d_L\alpha,\beta] + (-1)^{|\alpha|-1}\, [\alpha,d_L\beta]
\end{equation}
For more details on Lie bialgebroids see \cite{LiuWeinsteinXu}, and for their relation to generalized complex structures see \cite{GrandiniPoonRolle} and \cite{Gualtieri2011}.

\begin{example}
If $L$ corresponds to a complex structure, then $d_L = \bar\del$.  We can find the differential for other generalized complex structures by using the following fact, from \cite{GrandiniPoonRolle}:
\end{example}

\begin{prop}\label{deformation of d_L}
Let $L_\epsilon$ be an integrable deformation of a generalized complex structure $L$ by $\epsilon \in \^ ^2 L^*$.  As per our convention, we identify both $L^*$ and $L_\epsilon^*$ with $\bar{L}$, and thus identify their respective differential complexes as sets.  Then for $\sigma \in \Gamma(\^ ^k L^*)$,
$$d_{L_\epsilon} \sigma = d_L \sigma + [\epsilon, \sigma].$$
\end{prop}

\begin{example}\label{d_L for Poisson}
Thus, the differential on $\Gamma(\^ ^k L^*)$ coming from a holomorphic Poisson structure $\beta \in \Gamma(\^ ^2 T_{1,0})$ is just
$$d_{L_\beta} = \bar\del + d_\beta,$$
where $d_\beta$ is the usual Poisson differential $[\beta,\cdot]$.
\end{example}

\subsection{Integrability and the Maurer-Cartan equation}\label{Maurer-Cartan section}

For a deformed structure $L_\epsilon$ to be integrable, $\epsilon$ must satisfy the Maurer-Cartan equation,
\begin{equation}\label{MC equation}
d_L \epsilon + \frac{1}{2}[\epsilon,\epsilon] = 0
\end{equation}

\begin{notn}
Suppose $L$ is the $+i$-eigenbundle for the generalized complex structure on $\C^n$ coming from the complex structure and that $L^* = T_{1,0} \dsum T^*_{0,1}$, as in Section \ref{L* convention}.  We may write
$$\^ ^2 L^* = (\^ ^2 T_{1,0}) \dsum (T_{1,0} \tens T^*_{0,1}) \dsum (\^ ^2 T^*_{0,1}).$$
If $\epsilon \in \Gamma(\^ ^2 L^*)$ is a deformation, we will write $\epsilon$ correspondingly as $\epsilon_1 + \epsilon_2 + \epsilon_3$, where $\epsilon_1$ is a bivector field, $\epsilon_2 \in \Gamma(T_{1,0} \tens T^*_{0,1})$, and $\epsilon_3$ is a 2-form.
\end{notn}

Then the Maurer-Cartan condition (\ref{MC equation}) on $\epsilon$ splits into four equations:
\begin{align}
\^ ^3 T_{1,0}  & \quad:\quad  [\epsilon_1,\epsilon_1] = 0  \label{MC1}\\
\^ ^2 T_{1,0} \tens T^*_{0,1}  & \quad:\quad  [\epsilon_1,\epsilon_2] + \bar\del\epsilon_1 = 0  \label{MC2}\\
T_{1,0} \tens \^ ^2 T^*_{0,1}  & \quad:\quad  \frac{1}{2} [\epsilon_2,\epsilon_2] + [\epsilon_1,\epsilon_3] + \bar\del\epsilon_2 = 0  \label{MC3}\\
\^ ^3 T^*_{01,}  & \quad:\quad  [\epsilon_2,\epsilon_3] + \bar\del\epsilon_3 = 0  \label{MC4} 
\end{align}

\begin{rem}\label{bivector implies holomorphic}
By \eqref{MC1}, $\epsilon_1$ always satisfies the Poisson condition.  If $\epsilon_2=0$ then, by \eqref{MC2}, $\epsilon_1$ is also holomorphic.  Therefore, to say that an integrable deformation $\epsilon$ is holomorphic Poisson is the same as to say that $\epsilon_2$ and $\epsilon_3$ vanish, that is, that $\epsilon$ is just a bivector.
\end{rem}

\subsection{Generalized vector fields and generalized flows}\label{generalized diffeomorphisms}

In this section we discuss how generalized vector fields integrate to 1-parameter families of generalized diffeomorphisms, and how these act on deformations of generalized complex structures.

\begin{defn}\label{generalized diffeomorphism}
As we mentioned earlier \emph{generalized diffeomorphism} $\Phi : T\dsum T^* \to T\dsum T^*$, also called a \emph{Courant isoomorphism}, is an isomorphism of $T\dsum T^*$ (covering some diffeomorphism) which respects the Courant bracket, the symmetric pairing, and the projection to the tangent bundle.

A \emph{$B$-transform} is a particular kind of generalized diffeomorphism: if $B:T\to T^*$ is a closed 2-form and $X+\xi\in T\dsum T^*$, then we say that
$e^B(X+\xi) = (1+B)(X+\xi) = X+\iota_X B+\xi.$

Another kind of generalized diffeomorphism is a plain diffeomorphism acting by pushforward (which means inverse pullback on the $T^*$ component).  $B$-transforms and diffeomorphisms together generate the generalized diffeomorphisms \cite{Gualtieri2011}, and thus we will typically identify a generalized diffeomorphism $\Phi$ with a pair $(B,\phi)$, where $B$ is a closed 2-form and $\phi$ is a diffeomorphism---by convention $\Phi$ acts first via the $B$-transform and then via pushforward by $\phi_*$.
\end{defn}

\begin{rem}\label{Courant composition}
Let $\Phi=(B,\phi)$ and $\Psi=(B',\psi)$ be generalized diffeomorphisms.  Then
$$\Phi \comp \Psi = (\psi^*(B) + B',\phi\comp\psi)  \quad\textnormal{and}\quad
\Phi^{-1} = (-\phi_*(B),\phi^{-1})$$
\end{rem}

\begin{defn}\textbf{Action on sections.}\; If $v : M \to TM \dsum T^*M$ is a section of $TM \dsum T^*M$ and $\Phi=(B,\phi)$ is a generalized diffeomorphism, then the \emph{pushforward} of $v$ by $\Phi$ is
$$\Phi_* v = \Phi \comp v \comp \phi^{-1}.$$
\end{defn}

\begin{defn}\label{generalized action on deformation}
\textbf{Action on deformations.}\; If $L_\epsilon$ is a deformation of generalized complex structure $L$, and $\Phi$ is a generalized diffeomorphism of sufficiently small 1-jet, then $\Phi(L_\epsilon)$ is itself a deformation of $L$, by some $\Phi\cdot\epsilon \in \Gamma(\^ ^2 L^*)$.  That is, $\Phi\cdot\epsilon$ is such that $L_{\Phi\cdot\epsilon} = \Phi(L_\epsilon)$.  In other words,
\begin{equation}
\Phi\left((1+\epsilon)L\right) = (1+\Phi\cdot\epsilon)L.
\end{equation}
\end{defn}
(For a more concrete formula for $\Phi\cdot\epsilon$, see Proposition \ref{action formula}.)

\begin{rem}\label{not pushforward}
In general, $\Phi\cdot\epsilon$ should \emph{not} be understood as a pushforward of the tensor $\epsilon$.  (In fact, $\Phi\cdot0$ may be nonzero!)  However, in the special case where $\Phi$ respects the initial generalized complex structure, i.e., where $\Phi(L)=L$, then indeed $\Phi\cdot\epsilon = \Phi_*(\epsilon)$ suitably interpreted.
\end{rem}

\begin{defn}\label{generalized flow}
A section $v \in \Gamma(T\dsum T^*)$ is called a \emph{generalized vector field}.  We say that $v$ generates the 1-parameter family $\Phi_{tv}$ of generalized diffeomorphisms, or that $\Phi_{tv}$ is the generalized flow of $v$, if for any section $\sigma\in \Gamma(T\dsum T^*)$,
\begin{equation}\label{tensor Courant derivative}
\left.\frac{d}{dt}\right|_{\tau=t}\left(\Phi_{\tau v}\right)_*\sigma = [v,\left(\Phi_{tv}\right)_*\sigma].
\end{equation}
\end{defn}

The flow thus defined is related to the classical flow of diffeomorphisms as follows:

Let $v = X + \xi$, where $X$ is a vector field and $\xi$ a 1-form.  If $X$ is small enough, or the manifold is compact, then it integrates to the diffeomorphism $\phi_X$ which is its time-$1$ flow.  Let
$$B_v = \int_0^1 \phi_{tX}^* (d\xi) dt.$$
Then $\Phi_v = (B_v, \phi_X)$ is the time-1 generalized flow of $v$.

\begin{rem}
If $X$ does not integrate up to time 1 from every point, then $\phi_X$, and thus $\Phi_v$, is instead defined on a subset of the manifold.  In this case, $\Phi_v$ is a \emph{local} generalized diffeomorphism.
\end{rem}

\begin{rem}
While \eqref{tensor Courant derivative} gives the derivative of a generalized flow acting by pushforward on a tensor, it does \emph{not} hold for derivatives of generalized flows acting by the deformation action of Definition \ref{generalized action on deformation}, as we see from Remark \ref{not pushforward}.
\end{rem}

The following is a corollary to \cite[Prop. 5.4]{Gualtieri2011}: 
\begin{lem}
If $0 \in \Gamma(\^ ^2 L^*)$ is the trivial deformation of $L$ and $v \in \Gamma(T\dsum T^*)$, then
$$\left.\frac{d}{dt}\Phi_{tv}\cdot 0\right|_{t=0} = d_L v^{0,1},$$
where $v^{0,1}$ is the projection of $v$ to $L^*$.
\end{lem}

Then combining this fact with Proposition \ref{deformation of d_L} we see that
\begin{prop}\label{infinitesimal flow}
If $\epsilon \in \Gamma(\^ ^2 L^*)$ is an integrable deformation of $L$, and $v \in \Gamma(T\dsum T^*)$, then
$$\left.\frac{d}{dt}\Phi_{tv}\cdot \epsilon\right|_{t=0} = d_L v^{0,1} + [\epsilon,v],$$
where $v^{0,1}$ is the projection of $v$ to $L^*$.
\end{prop}

\begin{rem}
Definition \ref{generalized flow} makes sense if $v$ is a real section of ${T\dsum T^*}$.  On the other hand, if $v \in \Gamma(T_\C \dsum T^*_\C)$ is complex, we may interpret $\Phi_v$ in the presence of an underlying generalized complex structure as follows.  $v$ decomposes into $v^{1,0}\in L$ plus $v^{0,1}\in\bar{L}$.  We see in Proposition \ref{infinitesimal flow} that the component in $L$ has no effect on the flow of deformations, therefore we define
$$\Phi_v := \Phi_{v^{0,1} + \overline{v^{0,1}}},$$
where $v^{0,1} + \overline{v^{0,1}}$ is now real.  Proposition \ref{infinitesimal flow} still holds.
\end{rem}

\section{The infinitesimal case}\label{infinitesimal case}

We would like to make precise and then prove the following rough statement: if $\epsilon$ is an infinitesimal deformation of a holomorphic Poisson structure on the closed unit ball $B_1 \subset \C^n$, then we may construct an infinitesimal flow by a generalized vector field $V$ which ``corrects'' the deformation so that it remains within the class of holomorphic Poisson structures.  This turns out to be a cohomological claim about the complex $(\^ ^\bullet L^*, d_L)$.  When we consider the full problem of finite deformations, this will still be approximately true in some sense, which will help us prove the Main Lemma.

\begin{rem}
We often speak of the ``closed unit ball in $\C^n$,'' or something like it.  To be clear: since we are using $\sup$--norms rather than Euclidean norms (as is made explicit in Section \ref{norm section}), this is the same thing as the \emph{polydisc}, ${(D_1)^n \subset \C^n}$.
\end{rem}

\subsubsection*{Integrability of infinitesimal deformations}

Suppose that $\epsilon_t$ is a one-parameter family of deformations of $L$.  Differentiating equation \eqref{MC equation} by $t$, we get that
$$d_L \dot\epsilon_t + [\epsilon_t,\dot\epsilon_t] = 0$$
If $\epsilon_0 = 0$, then we have the condition $d_L \dot\epsilon_0 = 0$.  That is, an infinitesimal deformation of $L$ must be $d_L$-closed.

Thus we make precise the statement in the opening paragraph of this section:
\begin{prop}\label{infinitesimal correction}
Suppose that $L$ is the $+i$-eigenbundle corresponding to a holomorphic Poisson structure $\beta$ on $B_1\subset\C^n$, and suppose that $\epsilon \in \Gamma(\^ ^2 L^*)$ satisfies $d_L \epsilon = 0$.  Then there exists $V(\beta,\epsilon) \in \Gamma(L^*)$ such that $\epsilon + d_L V(\beta,\epsilon)$ has only a bivector component.
\end{prop}

\begin{proof}
As in Section \ref{Maurer-Cartan section}, we write $\epsilon = \epsilon_1 + \epsilon_2 + \epsilon_3$ where the terms are a bivector field, a mixed co- and contravariant term, and a 2-form respectively.  The closedness condition, $(\bar\del + d_\beta)\epsilon = 0$ (as per Example \ref{d_L for Poisson}), may be decomposed according to the co- and contravariant degree.

For example, we have $\bar\del\epsilon_3 = 0$.  Since $\bar\del$-cohomology is trivial on the ball $B_1$, there exists a $(0,1)$-form $P\epsilon_3$ such that $\bar\del P\epsilon_3 = \epsilon_3$.  $-P\epsilon_3$ will be one piece of $V(\beta,\epsilon)$.

Another component of the closedness condition is $\bar\del\epsilon_2 + d_\beta\epsilon_3 = 0$.  Then
\begin{eqnarray*}
\bar\del(\d_\beta P\epsilon_3 - \epsilon_2) &=& \bar\del\d_\beta P\epsilon_3 + \d_\beta\epsilon_3 \\
&=& \bar\del\d_\beta P\epsilon_3 + \d_\beta\bar\del P\epsilon_3
\end{eqnarray*}
But $\bar\del$ and $\d_\beta$ anticommute, so this is $0$, i.e., $\d_\beta P\epsilon_3 - \epsilon_2$ is $\bar\del$-closed.  Therefore it is $\bar\del$-exact, and there exists some $(1,0)$-vector field $P(\d_\beta P\epsilon_3 - \epsilon_2)$ such that $\bar\del P(\d_\beta P\epsilon_3 - \epsilon_2) = \d_\beta P\epsilon_3 - \epsilon_2$.  Let
\begin{equation}\label{construct V}
V(\beta,\epsilon) = P(\d_\beta P\epsilon_3 - \epsilon_2) - P\epsilon_3
\end{equation}
Then
$$(\bar\del + \d_\beta)V(\beta,\epsilon) = \d_\beta P(\d_\beta P\epsilon_3 - \epsilon_2) - \epsilon_2 - \epsilon_3,$$
where $\d_\beta P(\d_\beta P\epsilon_3 - \epsilon_2)$ is a section of $\^ ^2 T_{1,0}$.  Therefore
$$\epsilon + d_L V(\beta,\epsilon) \;\in\; \Gamma(\^ ^2 T_{1,0})$$
\end{proof}

\subsection{The $\bar\del$ chain homotopy operator}

The non-constructive step in the proof of Proposition \ref{infinitesimal correction} is the operation $P$ which gives $\bar\del$-primitives for sections of $(T_{1,0} \tens T^*_{0,1}) \dsum \^ ^2 T^*_{0,1}$.  Fortunately, in \cite{NijenhuisWoolf} Nijenhuis and Woolf give a construction of such an operator and provide norm estimates for it.

\begin{prop}\label{existence of P}
For a closed ball $B_r \subset \C^n$, there exists a linear operator $P$ such that for all $i,j\geq0$,
$$P : \Gamma\left(\left(\^ ^i T_{1,0}\right) \tens \left(\^ ^{j+1} T^*_{0,1}\right)\right)
\to \Gamma\left(\left(\^ ^i T_{1,0}\right) \tens \left(\^ ^j T^*_{0,1}\right)\right)$$
such that
\begin{equation}\label{P error}
\bar\del P + P \bar\del = \Id.
\end{equation}
and such that the $C^k$-norms of $P$ satisfies the estimate, for all integers $k\geq0$,
$$\|P\epsilon\|_k \,\leq\, C\,\|\epsilon\|_k.$$
(See Section \ref{norm section} for our conventions on $C^k$ norms.)
\end{prop}

We note that $P$ is defined on \emph{all} smooth sections, not just $\bar\del$-closed sections.  But if $\bar\del\epsilon=0$, $\bar\del P\epsilon = \epsilon$ as desired.

\begin{proof}
For a $(0,j)$ form, $P$ is just the operator $T$ defined in \cite{NijenhuisWoolf}.  We don't give the full construction here (or the proofs of its properties), but we remark that it is built inductively from the case of a 1-form $f\,d\bar{z}$ on $\C$, for which
$$(T\, f\,d\bar{z}) (x) = \frac{-1}{2\pi i} \int_{B_r} \frac{f(\zeta)}{\zeta-x}\, d\zeta\^d\bar\zeta.$$

On the other hand, if $\epsilon$ is not a differential form, but rather is a section of $\left(\^ ^i T_{1,0}\right) \tens \left(\^ ^{j+1} T^*_{0,1}\right)$ for $i>0$, we may write
$$\epsilon = \sum_I \frac{d}{dz_I} \tens \epsilon_I,$$
where $I$ ranges over multi-indices, $\frac{d}{dz_I}$ is the corresponding basis multivector, and $\epsilon_I \in \Gamma\left(\^ ^{j+1} T^*_{0,1}\right)$.  Then $T$ is applied to each of the $\epsilon_I$ individually.

The estimate is also from \cite{NijenhuisWoolf}, and by construction of $P$ clearly also applies to mixed co- and contravariant tensors.
\end{proof}

$P$ as defined depends continuously on the radius, $r$, of the polydisc---that is, it doesn't commute with restriction to a smaller radius.  We say no more about this quirk except to note that it poses no problems for us (for example, with Theorem \ref{Miranda-Monnier-Zung}).

\subsection{Approximating the finite case with the infinitesimal solution}

We sketch how Proposition \ref{infinitesimal correction} roughly translates to the finite case (for details, see Lemma \ref{almost infinitesimal}):

We will be considering deformations $\epsilon = \epsilon_1 + \epsilon_2 + \epsilon_3$ of the complex structure on $B_r \subset \C^n$, which are close to being holomorphic Poisson; thus, $\epsilon_2$ and $\epsilon_3$ will be small and $\epsilon_1$ will almost be a holomorphic Poisson bivector.  We then pretend that $\epsilon_2 + \epsilon_3$ is a small deformation of the almost holomorphic Poisson structure $\beta = \epsilon_1$, and the argument for Proposition \ref{infinitesimal correction} goes through approximately.  Thus,

\begin{defn}\label{define V}
If $\epsilon \in \Gamma(\^ ^2 L^*)$, with the decomposition $\epsilon = \epsilon_1 + \epsilon_2 + \epsilon_3$ as in Section \ref{Maurer-Cartan section}.  Then let
$$V(\epsilon) = V(\epsilon_1,\epsilon_2 + \epsilon_3) = P([\epsilon_1, P\epsilon_3] - \epsilon_2 - \epsilon_3).$$
\end{defn}

In the above construction, we apply $P$ to sections which are not quite $\bar\del$-closed, so it will not quite yield $\bar\del$-primitives; this error is controlled by equation \eqref{P error}.  Furthermore, we can no longer say that $[\epsilon_1,\cdot]$ and $\bar\del$ anticommute; this error will be controlled by the bialgebroid property \eqref{bialgebroid}, with $d_L = \bar\del$, so that if $\theta \in \Gamma(\^ ^\bullet L^*)$ then
\begin{equation}\label{anticommute error}
\bar\del[\epsilon_1, \theta] = -[\epsilon_1, \bar\del \theta] + [\bar\del\epsilon_1, \theta].
\end{equation}

\section{SCI-spaces and the abstract normal form theorem}\label{SCI section}

We are trying to show that, near a point $p$, a generalized complex structure is equivalent to one in a special class of structures (the holomorphic Poisson structures).  As discussed in Section \ref{outline of proof}, this is achieved by iteratively applying a particular sequence of \emph{local} generalized diffeomorphisms to the initial structure, and then arguing that in the limit this sequence takes the initial structure to a special structure.  One difficulty is that at each stage we may have to restrict to a smaller neighbourhood of $p$.  Thus the iteration is not over a fixed space of deformations, but rather over a collection of spaces, one for each neighbourhood of $p$.

The technique for handling this difficulty comes from Conn \cite{Conn}, though we have adopted some of the formalism of Miranda, Monnier and Zung \cite{MonnierZung} \cite{MirandaMonnierZung}, with \cite[Section 6 and Appendices A and B]{MirandaMonnierZung} our main reference.  We adapt the definition of SCI-spaces---or ``scaled $C^\infty$'' spaces---SCI-groups and SCI-actions, with some changes which we discuss.  In particular, for simplicity we consider only the ``$C^\infty$'' part of the space (whereas in \cite{MirandaMonnierZung} $C^k$ sections are considered).  Hence, an SCI-space is a radius-parametrized collection of tame Frechet spaces.  To be precise:

\begin{defn}
An \emph{SCI-space} $\sr H$ consists of a collection of vector spaces $\sr{H}_r$ with norms $\|\cdot\|_{k,r}$---where $k\geq0$ (the \emph{smoothness} or \emph{derivative degree}) is in $\Z$ and $0<r\leq1$ (the \emph{radius}) is in $\R$---and for every $0<r'<r\leq1$ a linear \emph{restriction map}, $\pi_{r,r'} : \sr H_r \to \sr H_{r'}$.  Furthermore, the following properties should hold:
\begin{itemize}
\item If $r>r'>r''$ then $\pi_{r,r''}=\pi_{r,r'}\comp\pi_{r',r''}$.
\end{itemize}
If $f \in \sr{H}_r$ then, to abuse notation, we denote $\pi_{r,r'}(f) \in \sr H_{r'}$ also by $f$. Then,
\begin{itemize}
\item If $f$ in $\sr{H}$, $r'\leq r$ and $k'\leq k$, then
$$\|f\|_{k',r'} \leq \|f\|_{k,r}, \qquad \textnormal{(monotonicity)}$$
\end{itemize}
where if neither $f$ nor a restriction of $f$ is in $\sr H_r$ then we interpret $\|f\|_r = \infty$.  We take as the topology for each $\sr{H}_r$ the one generated by open sets in every norm.  We require that
\begin{itemize}
\item If a sequence in $\sr{H}_r$ is Cauchy for each norm $\|\cdot\|_k$ then it converges in $\sr{H}_r$.
\item At each radius $r$ there are \emph{smoothing operators}, that is, for each real $t>1$ there is a linear map
$$S_r(t) : \sr H_r \to \sr H_r$$
such that for any $p>q$ in $\Z^+$ and any $f$ in $\sr H_r$,
\begin{eqnarray}
\|S_r(t)f\|_{p,r} &\leq& C_{r,p,q} t^{p-q}\|f\|_{q,r} \quad\textnormal{and} \\
\|f-S_r(t)f\|_{q,r} &\leq& C_{r,p,q} t^{q-p}\|f\|_{p,r},
\end{eqnarray}
where $C_{r,p,q}$ is a positive constant depending continuously on $r$.
\end{itemize}

An \emph{SCI-subspace} $\sr S \subset \sr H$ consists of a collection of subspaces $\sr S_r \subset \sr H_r$ which themselves form an SCI-space under the induced norms, restriction maps and smoothing operators.  An \emph{SCI-subset} of $\sr H$ consists of a collection of subsets of the $\sr H_r$ which is invariant under the restriction maps. 
\end{defn}

\begin{example}\label{sections are SCI}
Let $V$ be a finite-dimensional normed vector space.  For each $0<r\leq1$, let $B_r \subset \R^n$ or $\C^n$ be the closed unit ball of radius $r$ centred at the origin (under the $\sup$-norm, this is actually a rectangle or polydisc), and let $\sr H_r$ be the $C^\infty$-sections of the trivial bundle $B_r \times V$, with $\|\cdot\|_{k,r}$ the $C^k$-$\sup$ norm.  Then the $\sr H_r$ and $\|\cdot\|_{k,r}$ form an SCI-space.
\end{example}

\begin{rem}
At a fixed radius $r$, $\sr H_r$ is a tame Frechet space.  There are constructions of smoothing operators in many particular instances (see, eg., \cite{Hamilton}).  The essential point is that $S_r(t)f$ is a smoothing of $f$, in the sense that its higher-derivative norms are controlled by lower norms of $f$; however, as $t$ gets larger, $S_r(t)f$ is a better approximation to $f$, but is less smooth.  As a consequence of the existence of smoothing operators, we have the \emph{interpolation inequality} (also see \cite{Hamilton}):
\end{rem}
\begin{prop}\label{interpolation inequality}
Let $\sr{H}$ be an SCI-space, let $0\leq l \leq m \leq n$ be integers, and let $r>0$.  Then there is a constant $C_{l,m,n,r}>0$ such that for any $f \in \sr{H}_r$,
$$\|f\|_m^{n-l} \leq C_{l,m,n,r}\, \|f\|_n^{m-l}\, \|f\|_l^{n-m}.$$
\end{prop}

\subsection{Notational conventions}

We will need to express norm estimates for members of SCI-spaces, that is, we will write SCI-norms into inequalities.  We develop some shorthand for this, which is similar to (but extends) the notation in \cite{MirandaMonnierZung}.

\subsubsection*{Spaces of sections}
If $E = B_1 \times V$ is a vector bundle over $B_1 \subset \R^n$ or $\C^n$, then by $\Gamma(E)$ we will always mean the SCI-space of local sections of $E$ near $0 \in \C^n$, as in Example \ref{sections are SCI}.
\subsubsection*{Radius parameters}
We will often omit the radius parameter when writing SCI-norms (but we will always include the degree).  The right way to interpret such notation is as follows: when the norms appear in an equation, the claim is that this equation holds for any common choice of radius where all terms are well-defined.  When the norms appear in an inequality, the claim is that the inequality holds for any common choice of radius $r$ for the lesser side of the inequality, with any common choice of radius $r' \geq r$ for the greater side of the inequality (for which all terms are well-defined).

For example, for $f\in\sr{H}$ and $g\in\sr{K}$,
$$\|f\|_k \leq \|g\|_{k+1}$$
means
$$\forall\; 0<r\leq r'\leq1,\;\textnormal{if}\,f\in\sr{H}_r\,\textnormal{and}\, g\in\sr{H}_{r'}\,\textnormal{then}\, \|f\|_{k,r}\leq\|g\|_{k+1,r'}$$
\begin{rem}
Since the norms are nondecreasing in radius, this convention is plausible.
\end{rem}
\subsubsection*{Constants}
Whenever it appears in an inequality, $C$ (or $C'$) will stand for a positive real constant, which may be different in each usage, and which may depend on the degree, $k$, of the terms, and continuously on the radius.
\subsubsection*{Polynomials}
Whenever the notation $$Poly(\|f_1\|_{k_1},\|f_2\|_{k_2},\ldots)$$ occurs, it denotes some polynomial in $\|f_1\|_{k_1}$, $\|f_2\|_{k_2}$, etc., with positive coefficients, which may depend on the degrees $k_i$ and continuously on the radius, and which may be different in each usage.  These polynomials will always occur as bounds on the greater side of an inequality, and it will not be important to know their exact form.
\subsubsection*{Leibniz polynomials}
Because they occur so often, we give special notation for a certain type of polynomial.  Whenever the notation
$$\LP(\|f_1\|_{k_1},\, \ldots,\, \|f_d\|_{k_d})$$
occurs it denotes a polynomial (with positive coefficients, which depends on the $k_i$ and continuously on the radius, and which may be different in each usage) such that each monomial term is as follows:

Each $\|f_i\|_\bullet$ occurs with degree at least $1$ (in some norm degree), and at most one of the $\|f_i\|$ has ``large'' norm degree $k_i$, while the other factors in the monomial have ``small'' norm degree $\fl{k_j/2}+1$, where $\fl{\,\cdot\,}$ denotes the integer part.

Equivalently, using the monotonicity in $k$ of $\|\cdot\|_k$, we can define $\LP$ using $Poly$ notation, as follows:
\begin{eqnarray*}
& & A(f_1, \ldots, f_d) \leq \LP(\|f_1\|_{k_1},\, \ldots,\, \|f_d\|_{k_d}) \qquad\textnormal{if and only if}\notag \\
& & \qquad A(f_1,\ldots,f_d) \leq {\sum_{i=1}^d \|f_i\|_{k_i}\,\times\, \|f_1\|_{\fl{k_1/2}+1}\, \ldots\, \widehat{\|f_i\|}_{\fl{k_i/2}+1}\, \ldots\cdot\, \|f_d\|_{\fl{k_d/2}+1}}\, \notag \\
& & \qquad\qquad \times\, {Poly(\|f_1\|_{\fl{k_1/2}+1},\, \ldots,\, \|f_d\|_{\fl{k_d/2}+1})},
\end{eqnarray*}
where $\widehat{\|f_i\|}$ indicates this term is omitted from the product.  Given our definition, one can check that the following example is valid:
$$\|f\|_k\, \|g\|_{\fl{k/2}+2}\, +\, \|f\|_{\fl{k/2}+1}\, \|g\|_{k+2}\, \|g\|_{\fl{k/2}+2} \;\leq\; \LP(\|f\|_k,\|g\|_{k+2}).$$

\begin{rem}\label{Leibniz example}
A typical example of how such Leibniz polynomials arise is: to find the $C^k$-norm of a product of fields, we must differentiate $k$ times, applying the Leibniz rule iteratively. We get a polynomial in derivatives of the fields, and each monomial has at most one factor with more than ${\fl{k/2}+1}$ derivatives.  See Lemma \ref{bilinear estimate} for example, or \cite[II.2.2.3]{Hamilton} for a sharper estimate.
\end{rem}

We extend the definition to allow the entries in a Leibniz polyonmial to be polynomials themselves, eg.,
$$\|f\|_k\,\|h\|_{\fl{k/2}+1} \,+\, \|f\|_k\, \|g\|_{\fl{k/2}+1} + \|f\|_{\fl{k/2}+1}\, \|g\|_k \;\leq\; \LP\left(\|f\|_k,\, \|h\|_k+\|g\|_k\right).$$
In this case, we have used $\|h\|_k+\|g\|_k$ to indicate that not every monomial need have a factor of both $\|g\|$ and $\|h\|$.
\begin{lem}
Leibniz polynomials are closed under composition and addition, e.g.,
$$\LP(\LP(\|f\|_a,\, \|g\|_b),\, \|h\|_c) \,\leq\, \LP(\|f\|_a,\, \|g\|_b,\, \|h\|_c)$$
\end{lem}

\begin{rem}
The approach in \cite{Hamilton} is to study \emph{tame} maps between tame Frechet spaces.  To say that a map is bounded by a Leibniz polynomial in its arguments is similar to the tameness condition.  However, rather than adapt this framework to SCI-spaces, we do as in \cite{MonnierZung} and \cite{MirandaMonnierZung}, working directly with bounding polynomials.  Very recent work \cite{Marcut} undertakes to adapt this tameness framework to Conn-type arguments, with promising results.
\end{rem}

\begin{rem}\label{radius dependence}
As noted in \cite{MonnierZung} and elsewhere, whether the coefficients of the polynomials vary continuously with the radius, or are fixed, makes no difference to the algorithm of Theorem \ref{Miranda-Monnier-Zung}, which ensures that all radii are between $R/2$ and $R$, over which we can find a radius-independent bound on the coefficients.
\end{rem}

\subsection*{SCI-groups}
We will give a definition of a group-like structure modelled on SCI-spaces, which is used in \cite{MirandaMonnierZung} to model local diffeomorphisms about a fixed point (and in our case to model local generalized diffeomorphisms); but first we feel we should give a conceptual picture to make the definition clearer:

Elements of an SCI-group will be identified with elements of an SCI-space, and we use the norm structure of the latter to express continuity properties of the former.  However, we do not assign any special meaning to the linear structure of the SCI-space---in particular, the SCI model-space for an SCI-group should not be viewed as its Lie algebra in any sense.  Furthermore, group elements will be defined at given radii, and their composition may be defined at yet a smaller radius---the amount by which the radius shrinks should be controlled by $\|\cdot\|_1$ of the elements (usually interpreted as a bound on their first derivative) and a fixed parameter for the group.

\begin{defn}\label{SCI group}
An \emph{SCI-group $\sr{G}$ modelled on an SCI-space $\sr{W}$} consists of elements which are formal sums
$$\phi = \Id + \chi,$$
where $\chi \in \sr{W}$, together with a \emph{scaled product} defined for some pairs in $\sr{G}$, i.e.:

There is a constant $c>1$ such that if $\phi$ and $\psi$ are in $G_r$ for some $r$ and
$$\|\phi-\Id\|_{1,r} \leq 1/c,$$
then,

(a) the \emph{product} $\psi \cdot \phi \in \sr G_{r'}$ is defined, where $r'=r(1-c\|\phi-\Id\|_{1,r})$; furthermore, the product operation commutes with restriction, and is associative modulo necessary restrictions, and

(b) there exists a \emph{scaled inverse} $\phi^{-1} \in \sr{G}_{r'}$ such that $\phi\cdot\phi^{-1} = \phi^{-1}\cdot\phi = \Id$ at radius $r''=r'(1-c\|\phi-\Id\|_{1,r})$.

Furthermore, for $k\geq1$ the following continuity conditions should hold:
\begin{eqnarray}
\|\psi^{-1}-\phi^{-1}\|_k &\leq& \LP(\|\psi-\phi\|_k,\,1+\|\phi-\Id\|_k) \label{inverse estimate}, \\[3pt]
\|\phi\cdot\psi - \phi\|_k &\leq& \LP(\|\psi-\Id\|_k,\, 1+\|\phi-\Id\|_{k+1}) \label{composition estimate 1} \\[3pt]
\textnormal{and} \qquad
\|\phi\cdot\psi - \Id\|_k &\leq& \LP(\|\psi-\Id\|_k + \|\phi-\Id\|_k). \label{composition estimate 2}
\end{eqnarray}
(As per the notational convention, these inequalities are taken at precisely those radii for which they make sense.)
\end{defn}

\begin{example}\label{local diffeo are SCI}
As in Example \ref{sections are SCI}, for each $0<r\leq1$ let $B_r \subset \R^n$ be the closed unit ball of radius $r$ centred at the origin, and let $\sr{W}_r$ be the space of $C^\infty$-maps from $B_r$ into $\R^n$ fixing the origin.  If $\chi$ is such a map, then by $\phi = \Id + \chi$ we mean the sum of $\chi$ with the identity map; then $\Id + \sr{W}_r$ forms an SCI-group under composition for some constant $c>1$.  These are the local diffeomorphisms. (See Lemma \ref{functions are SCI} and \cite{Conn} for details.)
\end{example}

\begin{rem}
Our definition of SCI-group is a bit different than that appearing in \cite{MirandaMonnierZung}, our source for this material.  Our continuity conditions look different---though, ignoring terms of norm degree $\fl{k/2}+1$, our conditions imply those in \cite{MirandaMonnierZung}.  (See Remark \ref{justify changes} for more on this.)
\end{rem}

\begin{defn}\label{SCI action}
A \emph{left (resp.\ right) SCI-action} of an SCI-group $\sr G$ on an SCI-space $\sr H$ consists of an operation
$$\phi\,\cdot : f \to \phi\cdot f \in \sr H_{r'}$$
for $\phi \in \sr G_r$ and $f \in \sr H_r$, which is defined whenever $r' \leq (1-c\|\phi-\Id\|_{1,r})r$ for some constant $c>1$, such that the following hold: the operation should commute with radius restriction, it should satisfy the usual left (resp. right) action law modulo radius restriction, and there should be some constant $s$ (called the \emph{derivative loss}) such that, for large enough $k$, for $\phi,\psi \in \sr{G}_r$ and $f,g \in \sr{H}_r$, the following continuity conditions hold:
\begin{align}
\|\phi\cdot f - \phi\cdot g\|_k & \;\leq\; 
 \LP(\|f-g\|_k,\, 1+\|\phi-\Id\|_{k+s}) \quad\textnormal{and} \label{action estimate 1} \\[4pt]
\|\psi\cdot f - \phi\cdot f\|_k & \;\leq\; 
 \LP\left(1+\|f\|_{k+s},\, \|\psi-\phi\|_{k+s},\, 1+\|\phi-\Id\|_{k+s}\right), \label{action estimate 2} 
\end{align}
(whenever these terms are well-defined).
\end{defn}

\begin{rem}
\eqref{action estimate 2} will ensure that if a sequence $\phi_1,\;\phi_2,\;\ldots$ converges, then so does $\phi_1\cdot f,\; \phi_2\cdot f,\; \ldots$.  Combining \eqref{action estimate 1} with \eqref{action estimate 2} for $f=0$ and $\psi=\Id$, we get another useful inequality,
\begin{equation}\label{nonlinear action estimate}
\|\phi\cdot g\|_k \;\leq\; \LP(\|g\|_k \,+\, \|\phi-\Id\|_{k+s})
\end{equation}
\end{rem}

\begin{rem}\label{action nonlinearity}
If the action is linear, we may simplify to equivalent hypotheses: we may discard $g$ entirely in \eqref{action estimate 1}, and, since each term will be first order in norms of $f$, we may replace $1+\|f\|_k$ with $\|f\|_k$ in \eqref{action estimate 2}; furthermore, in both estimates the polynomials will not have higher powers of $\|f\|$.  In \cite{MirandaMonnierZung}, only linear (and, in some sense, \emph{affine}) SCI-actions are considered.

Even considering this difference, our definition is a bit stronger than in \cite{MirandaMonnierZung}---as per our definition of Leibniz polynomials, $\LP$, we do not permit more than one factor of high norm degree in each monomial.
\end{rem}

\begin{example}\label{pullbacks are SCI}
The principal example of SCI-actions are local diffeomorphisms (Example \ref{local diffeo are SCI}) acting by pushforward or pullback on tensors, with derivative loss $s=1$.  See Section \ref{pushforwards and pullbacks} for details.
\end{example}

\subsection{Abstract normal form theorem}

The following theorem is adapted from \cite[Thm. 6.8]{MirandaMonnierZung}, with some changes, which mostly relate to the need to generalize to nonlinear actions.  After the statement of the theorem, we give the interpretation of each SCI-space and map named in the theorem, as it applies to our situation---this interpretation is a more or less essential reference for the reader trying to parse the theorem---and then we show how the theorem may be used to prove our Main Lemma.  Finally, we address the differences between the theorem as we have presented it and as it appears in \cite{MirandaMonnierZung}.  An early prototype of this theorem is in \cite{MonnierZung}.

\begin{thm}\label{Miranda-Monnier-Zung}[MMZ]
Let $\sr{T}$ be an SCI-space, $\sr{F}$ an SCI-subspace of $\sr{T}$, and $\sr{I}$ a subset of $\sr{T}$ containing $0$.  Denote $\sr{N} = \sr{F} \cap \sr{I}$.  Let $\pi : \sr{T} \to \sr{F}$ be a projection commuting with restriction, and let $\zeta = Id - \pi$.  Suppose that, for all $\epsilon \in \sr{T}$, and all $k \in \N$ sufficiently large,
\begin{equation}\label{MMZ1}
\|\zeta(\epsilon)\|_k \leq \LP(\|\epsilon\|_k).
\end{equation}

Let $\sr{G}$ be an SCI-group acting on $\sr{T}$, and let $\sr{G}^0 \subset \sr{G}$ be a closed subset of $\sr{G}$ preserving $\sr{I}$.

Let $\sr{V}$ be an SCI-space.  Suppose there exist maps
$$\sr{I} \oto{V} \sr{V} \oto{\Phi} \sr{G}^0$$
(with $\Phi(v)$ denoted $\Phi_v$) and $s \in \N$ such that, for every $\epsilon \in \sr{I}$, every $v, w \in \sr{V}$, and for large enough $k$,
\begin{eqnarray}
\|V(\epsilon)\|_k &\leq& \LP(\|\zeta(\epsilon)\|_{k+s},\, 1+\|\epsilon\|_{k+s}) \label{MMZ2} \\[4pt]
\|\Phi_v - Id\|_k &\leq& \LP(\|v\|_{k+s}), \label{MMZ3} \quad\mathnormal{and} \\[4pt]
\|\Phi_v\cdot \epsilon - \Phi_w\cdot \epsilon\|_k &\leq& \LP(\|v - w\|_{k+s},\, 1 + \|v\|_{k+s} + \|w\|_{k+s}+\|\epsilon\|_{k+s}) \notag \\
& & \;+\; \LP\left((\|v\|_{k+s} + \|w\|_{k+s})^2,\, 1+\|\epsilon\|_{k+s}\right) \label{MMZ4}
\end{eqnarray}

Finally, suppose there is a real positive $\delta$ such that for any $\epsilon\in\sr{I}$,
\begin{equation}
\|\zeta(\Phi_{V(\epsilon)}\, \cdot\, \epsilon)\|_k \leq \|\zeta(\epsilon)\|_{k+s}^{1+\delta}\, Poly\left(\|\epsilon\|_{k+s}, \|\Phi_{V(\epsilon)}-\Id\|_{k+s}, \|\zeta(\epsilon)\|_{k+s}, \|\epsilon\|_k\right) \label{MMZ5}
\end{equation}
where in this case the degree of the polynomial in $\|\epsilon\|_{k+s}$ does not depend on $k$.

Then there exist $l \in \N$ and two constants $\alpha>0$ and $\beta>0$ with the following property: if $\epsilon \in \sr{I}_R$ such that $\|\epsilon\|_{2l-1,R} < \alpha$ and $\|\zeta(\epsilon)\|_{l,R} < \beta$, there exists $\psi \in \sr{G}^0_{R/2}$ such that $\psi \cdot \epsilon \in \sr{N}_{R/2}$.
\end{thm}

\begin{rem}\label{interpretation of theorem}
In our case, the interpretation of the terms in this theorem will be as follows:
\begin{itemize}
\item $\sr{T}$ will be the space of deformations, $\Gamma(\^ ^2 L^*)$, of the standard generalized complex structure on $\C^n$.
\item $\sr{F} \subset \sr{T}$ will be the space of $(2,0)$-bivectors, the ``normal forms'' without the integrability condition---thus $\zeta(\epsilon) = \epsilon_2 + \epsilon_3$ is the non-bivector part of $\epsilon$, which we seek to eliminate.
\item $\sr{I}$ will be the integrable deformations, and thus $\sr{N} = \sr{F} \cap \sr{I}$ will be the holomorphic Poisson bivectors, i.e., the ``normal forms.''
\item $V$ produces a generalized vector field from a deformation.  As per Definition \ref{define V}, we will take $V(\epsilon) = P([\epsilon_1, P\epsilon_3] - \epsilon_2 - \epsilon_3)$.
\item $\sr{G}=\sr{G}^0$ will be the local generalized diffeomorphisms fixing the origin, acting on deformations as in Definition \ref{generalized action on deformation}, and $\Phi_v \in \sr{G}$ will be the time-1 flow of the generalized vector field $v$ as in Definition \ref{generalized flow}.
\end{itemize}

While estimates \eqref{MMZ1} through \eqref{MMZ4} in the hypotheses of the theorem may be understood as continuity conditions of some sort, estimate \eqref{MMZ5} expresses the fact that we have the ``correct'' algorithm, that is, each iteration will have a ``quadratically'' small error.
\end{rem}

\subsection{Proving the Main Lemma}\label{prove main lemma}
In Section \ref{verifying SCI} we verify that local generalized diffeomorphisms form a closed SCI-group, and that they act by SCI-action on the deformations.  In Section \ref{verifying hypotheses} we show that the other hypotheses of Theorem \ref{Miranda-Monnier-Zung}, estimates \eqref{MMZ1} through \eqref{MMZ5}, hold true for the interpretation above.  Thus, the theorem applies, and we conclude the following: if $\epsilon$ is a smooth, integrable deformation of the standard generalized complex structure in a neighbourhood of the origin in $\C^n$, and if $\|\epsilon\|_k$ is small enough (for some $k$ given by the theorem), then there is a local generalized diffeomorphism $\Psi$ fixing the origin such that $\zeta(\Psi\cdot\epsilon)=0$.  Then the Maurer-Cartan equations \eqref{MC1} and \eqref{MC2} tell us that $\Psi\cdot\epsilon$ is a holomorphic Poisson bivector, and thus the Main Lemma is proved.

\subsection{Sketch of the proof of Theorem \ref{Miranda-Monnier-Zung}}\label{sketch of MMZ proof}
The proof of Theorem \ref{Miranda-Monnier-Zung} is essentially in \cite[Appendix 1]{MirandaMonnierZung}, with the idea of the argument coming from \cite{Conn}.  Rather than give a full proof of our version, we give a rough sketch of the argument as it appears in \cite{MirandaMonnierZung} and, for the reader who wishes to verify in detail, in Remark \ref{justify changes} we justify the changes we have made from \cite{MirandaMonnierZung}.

We are given $\epsilon=\epsilon^0 \in \sr{I}_R$ and will construct a sequence $\epsilon^1,\epsilon^2,\ldots$.  We choose a sequence of smoothing parameters $t_0,t_1,t_2,\ldots$, with $t_0>1$ (determined by the requirements of the proof) and $t_{d+1}=t_d^{3/2}$.  Then for $d>0$ let $v_d = S_{t_d} V(\epsilon^d)$, where $S_{t_d}$ is the smoothing operator, let $\Phi_{d+1} = \Phi_{v_d}$, and let $\epsilon^{d+1}=\Phi_{d+1}\cdot \epsilon^d$.  The generalized vector field $V(\epsilon^d)$ is smoothed before taking its flow $\Phi_{d+1}$ so that we have some control over the loss of derivatives at each stage.

If $\|\epsilon\|_{2l-1}$ and $\|\zeta(\epsilon)\|_l$ are small enough, for certain $l$, and if $t_0$ is chosen carefully, then it will follow (after hard work!) that the $\|\Phi_d-\Id\|_k$ approach zero quickly and the corresponding radii have lower bound $R/2$; by continuity properties of SCI-groups and -actions, the compositions ${\Psi_{d+1} = \Phi_{d+1}\cdot\Psi_d}$ will have a limit, $\Psi_\infty$, and the $\epsilon^d$ will have a limit, ${\epsilon^\infty = \Psi_\infty\cdot \epsilon}$.  Furthermore, it will follow that ${\zeta(\epsilon^\infty) = \lim\zeta(\epsilon^d) = 0}$, so $\epsilon^\infty\in\sr{N}$.

The ``hard work'' from which the above facts follow is in two inductive lemmas.  The first fixes a norm degree, $l$, and an exponent, $A>1$, (determined by requirements of the proof) and proves inductively that for all $d\geq0$,
$$\begin{array}{llcl}
(1_d) & \|\Phi^{d+1}-\Id\|_{l+s} &<& t_d^{-1/2} \\
(2_d) & \|\epsilon^d\|_l &<& C\frac{d+1}{d+2} \\
(3_d) & \|\epsilon^d\|_{2l-1} &<& t_d^A \\
(4_d) & \|\zeta(\epsilon^d)\|_{2l-1} &<& t_d^A \\
(5_d) & \|\zeta(\epsilon^d)\|_l &<& t_d^{-1}
\end{array}$$
The second lemma uses the first to prove by induction on $k$ that, for all $k\geq l$, there is $d_k$ large enough such that for all $d\geq d_k$,
 $$\begin{array}{llcl}
(i) & \|\Phi^{d+1}-\Id\|_{k+s+1} &<& C_k t_d^{-1/2} \\
(ii) & \|\epsilon^d\|_{k+1} &<& C_k\frac{d+1}{d+2} \\
(iii) & \|\epsilon^d\|_{2k-1} &<& C_k t_d^A \\
(iv) & \|\zeta(\epsilon^d)\|_{2k-1} &<& C_k t_d^A \\
(v) & \|\zeta(\epsilon^d)\|_k &<& C_k t_d^{-1}
\end{array}$$

Given this setup, the proofs simply proceed in order through ${1_d,\ldots,5_d}$ and ${i,\ldots,iv}$ by application of the hypotheses of Theorem \ref{Miranda-Monnier-Zung}, the continuity conditions for SCI-groups and SCI-actions, and the property of the smoothing operators.

\begin{rem}\label{justify changes}
The differences between the theorem as we have presented it and as it appears in \cite{MirandaMonnierZung} include notational and other minor changes, which we do not remark upon, and changes to the estimates coming from the nonlinearity of our action, which we now justify, with reference to \cite[Sections 6.2 and 7]{MirandaMonnierZung}:

First we remark that, for the SCI definitions and for the hypotheses of the abstract normal form theorem, our estimates imply theirs if any instance of $\|\epsilon\|_p$, $\|\zeta(\epsilon)\|_p$ or $\|\Phi-\Id\|_p$ in \cite{MirandaMonnierZung} is replaced with the nonlinear $\LP(1+\|\epsilon\|_p)$, $\LP(\|\zeta(\epsilon)\|_p)$ or $\LP(\|\Phi-\Id\|_p)$ respectively.  That is, in several instances their estimates require that the bound be linear in one of the above quantities, whereas we have allowed extra factors of $\|\cdot\|_{\fl{p/2}+1}$; furthermore, we have allowed certain extra terms which are zero-order in $\|\epsilon\|_p$.

But this is not a problem---the estimates are \emph{locally equivalent} (we will be precise), and thus are valid over the sequence defined above.  To see why, we note that, as stated in the sketch of the proof above, in the two subsidiary lemmas, $\|\epsilon^d\|_p$ only appears with $p\leq2l-1$ in the first lemma or $p\leq2k-1$ in the second.  But then
$$\LP(1+\|\epsilon^d\|_p) \,=\, (1+\|\epsilon\|_p)\,Poly(\|\epsilon^d\|_{\fl{p/2}+1}) \,\leq\, (1+\|\epsilon\|_p)\,Poly(\|\epsilon^d\|_l)$$
(and respectively for $k$.)  But the inductive hypothesis has that $\|\epsilon^d\|_l$ (resp. $\|\epsilon^d\|_k$) is bounded by a constant, so this extra polynomial factor does no harm.  Similarly, the extra factors in $\LP\left(\|\zeta(\epsilon^d)\|_p\right)$ and $\LP\left(\|\Phi^d-\Id\|_p\right)$ are vanishingly small by the inductive hypothesis.

The remaining concern, then, is for $1+\|\epsilon\|_p$ in place of $\|\epsilon\|_p$.  This is already dealt with implicitly in the \emph{affine} version of the theorem in \cite{MirandaMonnierZung}: the space $\sr{T}$ may be embedded affinely in $\C \dsum \sr{T}$, by $\epsilon \mapsto (1,\epsilon)$, with the norm $\|(1,\epsilon)\|_p = 1 + \|\epsilon\|_p$.  The constraint $\alpha$ in the hypothesis, $\|\epsilon\|_{2l-1}\leq\alpha$, in the original theorem can always be chosen greater than $1$, so in the affine context we simply require that $\|\epsilon\|_{2l-1} \,\leq\, \alpha'=\alpha-1$.
\end{rem}

\section{Verifying the SCI estimates}\label{verifying SCI}

In this section we explain how the particular objects named in Remark \ref{interpretation of theorem} satisfy the SCI definitions.

\subsection{Norms}\label{norm section}

As promised, to be precise, we state our conventions for $C^k$ $\sup$-norms.

\begin{defn}\label{norms 1}
Let $X \in \R^q$ or $\C^q$.  $X_i$ is the $i$-th component.  Then let
$$\|X\| = \sup_i |X_i|.$$
Similarly, if $A=[a_{ij}]$ is an $n\times n$ matrix, let $\|A\| = \sup_{i,j} |a_{ij}|$.
\end{defn}
\begin{rem}\label{matrix inverse estimate}
Comparing our matrix norm to the operator norm ${\|\cdot\|_{op}}$, we have
$$\|A\| \,\leq\, \|A\|_{op} \,\leq\, n\,\|A\|.$$
Then if $\|A-\Id\| \,\leq\, \frac{1}{2n}$, $A$ is invertible and
$$\|A^{-1}\| \,\leq\, 2.$$
\end{rem}

\begin{defn}\label{norms}
Suppose now that $f$ is a vector-valued function, $f : U \to V$, where $U \subset \R^n$ or $\C^n$ and $V$ is a normed finite-dimensional vector space.  Then let
$$\|f\|_0 = \sup_{x\in U} \|X(x)\|.$$

Suppose furthermore that $f$ is smooth.  If $\alpha$ is a multi-index, then $f^{(\alpha)}$ is the corresponding higher-order partial derivative.  If $k$ is a non-negative integer, then $f^{(k)}$ is an array containing the terms $f^{(\alpha)}$ for $|\alpha|=k$.  Let
\begin{equation*}
\|f\|_k = \sup_{|\alpha| \leq k} \|f^{(\alpha)}\|_0 = \sup_{j \leq k} \|f^{(k)}\|_0.
\end{equation*}
\end{defn}

\begin{rem}
Ultimately, we will always be working over the manifold $\C^n$ or a subset thereof.  Using the standard trivialization of the tangent and cotangent bundles, Definitions \ref{norms 1} and \ref{norms} give us a nondecreasing family of norms, $\|\cdot\|_k$, on smooth, $C^\infty$--bounded tensor fields on subsets of $\C^n$.  This applies to generalized vector fields, $B$-fields, and higher rank tensors (including deformations in $\^ ^2 L^*$).  However, for technical reasons, we must use a slightly unusual norm for generalized diffeomorphisms:
\end{rem}

\begin{defn}
If $\Phi = (B,\phi)$ is a local generalized diffeomorphism over a subset of $\C^n$, then we usually only take norms of $\Phi - \Id = (B,\phi-\Id)$.  Considering $\phi-\Id$ as just a function from a subset of $\C^n$ to $\C^n$, let
\begin{equation}
\|\Phi-\Id\|_k = \sup(\|B\|_{k-1},\,\|\phi-\Id\|_k).
\end{equation}
(For the special case $k=0$, replace $k-1$ with $k$.) The difference in degree between $B$ and $\phi$ reflects the fact that $B$ acts on derivatives while $\phi$ acts on the underlying points of the manifold.
\end{defn}


\subsection{Pushforwards and pullbacks}\label{pushforwards and pullbacks}

As mentioned in Examples \ref{local diffeo are SCI} and \ref{pullbacks are SCI},
\begin{lem}\label{functions are SCI}
Local diffeomorphisms from the closed balls $B_r \subset \R^n$ to $\R^n$ fixing the origin form an SCI-group under composition (see Example \ref{local diffeo are SCI}) with constant $c=2n$.  Furthermore, the pullback action, $\phi^*f = f\comp\phi$, of a local diffeomorphism $\phi$ on a function $f:B_r\to\C^p$, is a right SCI-action, with derivative loss $s=1$.
\end{lem}
As we said earlier, the continuity/tameness estimates, \eqref{inverse estimate}, \eqref{composition estimate 1}, \eqref{composition estimate 2}, \eqref{action estimate 1} and \eqref{action estimate 2}, in our definitions of SCI-group and SCI-action are slightly different from those in \cite{MirandaMonnierZung}: for SCI-groups, \eqref{inverse estimate} is stronger, and \eqref{composition estimate 1} and \eqref{composition estimate 2} have the same first-order behaviour in each group element, while having possibly higher-order terms (but only in $\fl{k/2}+1$ norms).  For SCI-actions, \eqref{action estimate 1} and \eqref{action estimate 2} are nonlinear counterparts to the conditions in \cite{MirandaMonnierZung}.
\begin{proof}
The proof of Lemma \ref{functions are SCI}, including the existence of compositions and inverses at the correct radii and the various continuity estimates, is essentially in \cite{Conn} (and \cite{MonnierZung}, with minor differences as noted).  We show only the proof of \eqref{inverse estimate}---the SCI-group continuity estimate for inverses---since it gives the flavour of the proofs of the other estimates, and differs most significantly from \cite{MirandaMonnierZung}.

Let $\phi$ and $\psi$ be local diffeomorphisms.  We proceed by induction on the degree of the norm.  If $\alpha=(\alpha_1,\ldots,\alpha_n)$ is a multi-index, we denote the $\alpha$-order partial derivative $D_\alpha$.  Suppose that \eqref{inverse estimate} holds for degree less than $k$, that is, whenever $|\alpha| < k$,
$$\|D_\alpha(\phi^{-1}-\psi^{-1})\|_0 \,\leq\, \LP(\|\phi-\psi\|_{|\alpha|},1+\|\psi-\Id\|_{|\alpha|}).$$
(This certainly holds for $|\alpha|=0$, given the hypothesis that $\|\phi-\Id\|_1 \leq 1/2n$ and likewise for $\psi$.)

Now suppose that $|\alpha| = k$.  We have the trivial identity
\begin{equation}\label{e80}
0 \;=\; D_\alpha\left((\phi-\psi)\comp\phi^{-1}\right) + D_\alpha\left(\psi\comp\phi^{-1} - \psi\comp\psi^{-1}\right)
\end{equation}

To compute the derivative $D_\alpha(\psi\comp\phi^{-1} - \psi\comp\psi^{-1})$ at $x\in B_r$, we make repeated applications of the chain rule and Leibniz rule, so that we have a sum of terms each of which has the form, for some $|\beta|\leq|\alpha|$,
\begin{equation}\label{e81}
D_\beta\psi|_{\phi^{-1}(x)}\cdot Q_\beta\left(\phi^{-1}\right)|_x - D_\beta\psi|_{\psi^{-1}(x)}\cdot Q_\beta\left(\psi^{-1}\right)|_x,
\end{equation}
where $Q_\beta(\phi^{-1})$ is a polynomial expression in derivatives of $\phi^{-1}$ up to order $\left.|\alpha|+1-|\beta|\right.$ (and likewise for $Q_\beta(\psi^{-1})$).  We remark that each term like \eqref{e81} will have at most one factor with higher derivatives than $\fl{k/2}+1$. Equivalently, \eqref{e81} is
\begin{eqnarray*}
D_\beta\psi|_{\psi^{-1}(x)}\cdot \left.\left(Q_\beta(\phi^{-1}) - Q_\beta(\psi^{-1})\right)\right|_x
\;+\; D_\beta(\psi|_{\phi^{-1}(x)} - \psi_|{\psi^{-1}(x)})\cdot Q_\beta\left(\phi^{-1}\right)|_x
\end{eqnarray*}

When $|\beta|=1$, i.e., when $\beta=i$ is just a single index, this is
\begin{equation}\label{e002}
\frac{\del}{\del x_i} \psi|_{\psi^{-1}(x)} \cdot D_\alpha(\phi^{-1}-\psi^{-1})|_x \;+\; D_\beta(\psi|_{\phi^{-1}(x)} - \psi|_{\psi^{-1}(x)})\cdot D_\alpha\phi^{-1}|_x.
\end{equation}
Thus, we may solve \eqref{e80} for $D_\alpha(\phi^{-1}-\psi^{-1})|_x$ as follows: we collect all terms like the first term in \eqref{e002}, then by inverting the matrix $a_{ij}=\frac{\del}{\del x_i} \psi_j|_{\psi^{-1}(x)}$ and applying it to \eqref{e80}, we see that the result will be $D_\alpha(\phi^{-1}-\psi^{-1})|_x$ plus some other terms, each of which will be built from one or more of the following factors:
\begin{itemize}
\item a matrix inverse whose norm is bounded by $2$
$$(\textnormal{since}\; \left\|\frac{\del}{\del x_i} \psi_j|_{\psi^{-1}(x)}-\Id\right\| \leq \frac{1}{2n}).$$
\item derivatives of $\psi$ up to order $k$,
\item derivatives of $\phi^{-1}-\psi^{-1}$ up to order $k-1$,
\item derivatives of $\phi-\psi$ up to order $k$,
\item derivatives of $\phi^{-1}$ up to order $k$, and
\end{itemize}

What if we compute the norm of this solution for $D_\alpha(\phi^{-1}-\psi^{-1})|_x$?  In the special case where $\phi=\Id$, by applying the induction hypothesis and combining the Leibniz polynomials we obtain that $\|\psi^{-1}-\Id\|_k \,\leq\, \LP(\|\psi-\Id\|_k)$.   Returning to the general case, this gives us the bound on $\|\phi^{-1}-\Id\|_k$ that we need to complete the proof.
\end{proof}

\begin{lem}\label{vector bundle is SCI}
Let $E=B_r \times V$ be a trivial rank-$n$ vector bundle over the closed ball $B_r \subset \C^n$, for each $0<r\leq1$.  Then the the vector bundle automorphisms covering the identity, $\Aut(E)$, form an SCI-group with constant $c=2n$, and act by SCI-action on the sections, $\Gamma(E)$, with derivative loss $s=0$.
\end{lem}
If $\Aut(E)$ and $\Gamma(E)$ are treated as matrix- and vector-valued functions respectively, then the necessary estimates follow in a straightforward way from applications of the product rule, as in Remark \ref{Leibniz example}, using Remark \ref{matrix inverse estimate} for the inverse estimate \eqref{inverse estimate}.

The following lemma tells us that an action will be SCI if it is composed of SCI-actions in a certain sense.  In fact, we don't use any of the algebraic structure of actions.
\begin{lem}\label{composite action}
Let $\sr{A}$, $\sr{B}$, $\sr{G}$ and $\sr{H}$ be SCI-spaces, where $\sr{A}$, $\sr{B}$ and $\sr{G}$ each have a distinguished element $\Id$, let
$$\cdot:\sr{A}\times\sr{H} \to \sr{H} \quad\textnormal{and}\quad \cdot:\sr{B}\times\sr{H} \to \sr{H}$$
be operations satisfying estimates \eqref{action estimate 1} and \eqref{action estimate 2} with derivative loss $s_1$ and $s_2$ respectively (no other SCI-action structure is assumed), and let
$$\cdot:\sr{G}\times\sr{H} \to \sr{H}$$
be an operation such that, for each $\phi \in \sr{G}$ there are $\phi_A \in \sr{A}$ and $\phi_B \in \sr{B}$ (with $\Id_A=\Id$ and $\Id_B=\Id$) such that for each $h \in H$,
$$\phi\cdot h = \phi_A\cdot(\phi_B\cdot h).$$
Finally, suppose there is an $s_3$ such that for any $\phi,\psi \in \sr{G}$ and large enough $k$,
\begin{equation}\label{action comparison}
\|\phi_A-\psi_A\|_k \,\leq\, \LP(\|\phi-\psi\|_{k+s_3}) \quad\textnormal{and}\quad
\|\phi_B-\psi_B\|_k \,\leq\, \LP(\|\phi-\psi\|_{k+s_3}).
\end{equation}
Then the operation of $\sr{G}$ on $\sr{H}$ also satisfies estimates \eqref{action estimate 1} and \eqref{action estimate 2} with derivative loss $s_1+s_2+s_3$.
\end{lem}
\begin{proof}
If $\phi\in\sr{G}$ and $f,g \in \sr{H}$, we apply estimate \eqref{action estimate 1} for the actions of $\sr{A}$ and $\sr{B}$:
\begin{eqnarray*}
\|\phi\cdot f - \phi\cdot g\|_k &=& \|\phi_A\cdot(\phi_B\cdot f) - \phi_A\cdot(\phi_B\cdot g)\|_k \\
&\leq& \LP(\|\phi_B\cdot f - \phi_B\cdot g)\|_k,\, 1+\|\phi_A-\Id\|_{k+s_1}) \\
&\leq& \LP(\LP(\|f-g\|_k,\, 1+\|\phi_B-\Id\|_{k+s_1}),\, 1+\|\phi_A-\Id\|_{k+s_1})
\end{eqnarray*}
Composing the Leibniz polynomials and using \eqref{action comparison} for $\|\phi_A-\Id\|$ and $\|\phi_B-\Id\|$, we see that estimate \eqref{action estimate 1} holds for the action of $\sr{G}$, with derivative loss $s_1 + s_3$.

If $\phi,\psi \in \sr{G}$ and $f \in \sr{H}$, then
\begin{eqnarray}
\|\psi\cdot f - \phi\cdot f\|_k &=& \|\psi_A\cdot(\psi_B\cdot f) \,-\, \phi_A\cdot(\phi_B\cdot f)\|_k \notag \\
&\leq& \|\psi_A\cdot(\psi_B\cdot f) - \phi_A\cdot(\psi_B\cdot f)\|_k \label{e222} \\
& & \qquad\qquad +\; \|\phi_A\cdot(\psi_B\cdot f) - \phi_A\cdot(\phi_B\cdot f)\|_k. \label{e223}
\end{eqnarray}
Similarly to above, we apply estimate \eqref{action estimate 2} to line \eqref{e222} and estimate \eqref{action estimate 1} to line \eqref{e223}, and then vice versa, followed by the estimates \eqref{action comparison}, and we see that \eqref{action estimate 2} holds for the action of $\sr{G}$, with total derivative loss $s_1 + s_2 + s_3$.
\end{proof}

\begin{lem}\label{pushforward and pullback are SCI}
The action of local diffeomorphisms by pushforward or by pullback on tensors constitutes an SCI-action with derivative loss $s=1$.
\end{lem}
\begin{proof}
If $\phi:B_r\to\R^n$ is a local diffeomorphism with $\|\phi-\Id\|_1\leq1/2n$ and $v:B_r \to TB_r \iso B_r\times\R^n$ is a vector field, then the pushforward of $v$ by $\phi$ may be decomposed as
$$\phi_*v = (D\phi\cdot v)\comp\phi^{-1},$$
where the derivative $D\phi$ is treated as a matrix-valued function, acting on $v$ by fibrewise multiplication.  Similarly, if $\theta:B_r \to T^*B_r \iso B_r\times\R^n$ is a 1-form, then the pushforward of $\theta$ may be written
$$\phi_*\theta = ((D\phi^T)^{-1}\cdot\theta)\comp\phi^{-1},$$
where the $(D\phi^T)^{-1}$ is the matrix transpose and inverse at each point.  We may regard $D\phi\cdot v$ and $(D\phi^T)^{-1}\cdot\theta$ as functions from $B_r$ to $\R^n$, in which case precomposition by $\phi^{-1}$ acts by SCI-action with derivative loss $s=1$; and $D\phi$ and $(D\phi^T)^{-1}$ are automorphisms of the vector bundle $B_r \times \R^n$, and thus act by SCI-action with derivative loss $s=0$.  If $\psi$ is another local diffeomorphism then
$$\|D\psi-D\phi\|_{k-1} \,\leq\, \|\psi-\phi\|_k,$$
and
$$\|(D\psi^T)^{-1}-(D\phi^T)^{-1}\|_{k-1}  \,\leq\, \LP(\|D\psi-D\phi\|_{k-1}),$$
so by taking a degree-shifted norm on the $D\psi-D\phi$, we are in the case of Lemma \ref{composite action}.

A similar argument works for pullbacks, and for higher-rank tensors.
\end{proof}

\subsection{Estimates of generalized actions}

\begin{lem}
Local generalized diffeomorphisms on the balls $B_r\subset\C^n$ form an SCI-group.
\end{lem}
\begin{proof}
Recall (Definition \ref{generalized diffeomorphism}) that a local generalized diffeomorphism $\Phi$ may be represented $(B,\phi)$, where $B$ is a closed 2-form and $\phi$ is a local diffeomorphism.  If $\Psi=(B',\psi)$ is another local generalized diffeomorphism, then
$$\Phi\comp\Psi = (\psi^*B+B',\phi\comp\psi) \qquad\textnormal{and}\qquad \Phi^{-1} = (-(\phi^{-1})^*B,\phi^{-1}).$$
We already know that local diffeomorphisms form an SCI-group, and
$$r(1-c\|\Phi-\Id\|_{1,r}) \leq r(1-c\|\phi-\Id\|_{1,r}),$$
thus products and inverses exist at precisely the radii required in the definition.  Furthermore, estimates \eqref{inverse estimate}, \eqref{composition estimate 1} and \eqref{composition estimate 2} will be satisfied for the diffeomorphism term, $\phi$, thus we only need to check them for the $B$-field term.

\textbf{Estimate \eqref{inverse estimate}.}  We verify the bound for the $B$-field part of ${\Phi^{-1}-\Psi^{-1}}$. Recall that the norm degree is shifted for the $B$-field term.  Since pushforward is an SCI-action with derivative loss 1, we may use its SCI-action estimates in the proof:
\begin{eqnarray*}
& & \|(\phi^{-1})^*B - (\psi^{-1})^*B'\|_{k-1} \\
&=& \|\phi_*B - \psi_*B'\|_{k-1} \\
&\leq& \|\phi_*B-\phi_*B')\|_{k-1} + \|\phi_*B' - \psi_*B'\|_{k-1} \\
&\leq& \LP(\|B-B'\|_{k-1},\, 1+\|\phi-\Id\|_k) \;+\; \LP(\|B'\|_k,\, \|\phi-\psi\|_k,\, \|\psi-\Id\|_k) \\
&\leq& \LP(\|\Phi-\Psi\|_k,\, 1+\|\Phi-\Id\|_k) \;+\; \LP(\|\Psi-\Id\|_{k+1},\, \|\Phi-\Psi\|_k,\, \|\Psi-\Id\|_k)
\end{eqnarray*}
We use $1+\|\Phi-\Id\|_k \,\leq\, 1+\|\Psi-\Id\|_k+\|\Phi-\Psi\|_k$ and combine the Leibniz polynomials to get estimate \eqref{inverse estimate}.

\textbf{Estimate \eqref{composition estimate 1}.}  Now we verify the bound for the $B$-field part of ${\Phi\comp\Psi-\Phi}$.  We use estimate \eqref{action estimate 2} for pullbacks on the second line:
\begin{eqnarray*}
\|\psi^*B+B'-B\|_{k-1} &\leq& \|\psi^*B-B\|_{k-1} \,+\, \|B'\|_{k-1} \\
&\leq& \LP(1+\|B\|_k,\, \|\psi-\Id\|_k) \,+\, \|B'\|_{k-1} \\
&\leq& \LP(1+\|\Phi-\Id\|_{k+1},\, \|\Psi-\Id\|_k) \,+\, \|\Psi-\Id\|_k \\
&\leq& \LP(1+\|\Phi-\Id\|_{k+1},\, \|\Psi-\Id\|_k)
\end{eqnarray*}

\textbf{Estimate \eqref{composition estimate 2}.}  Finally, we verify the bound for the $B$-field part of ${\Phi\comp\Psi-\Id}$.  We use estimate \eqref{action estimate 1} on the second line:
\begin{eqnarray*}
\|\psi^*B + B'-0\|_{k-1} &\leq& \|\psi^*B\|_{k-1} \,+\, \|B'\|_{k-1} \\
&\leq& \LP(1+\|\psi-\Id\|_k,\, \|B\|_{k-1}) \,+\, \|B'\|_{k-1} \\
&\leq& \LP(1+\|\Psi-\Id\|_k,\, \|\Phi-\Id\|_k) \,+\, \|\Psi-\Id\|_k
\end{eqnarray*}
and the result follows.
\end{proof}

\begin{rem}
Finally, we must show that the pairs $(\phi,B)$ which represent local generalized diffeomorphisms are closed in $\textrm{Diff} \times \Omega^2$, i.e., that they are complete.  We consider a $C^\infty$-convergent sequence of local generalized diffeomorphisms,
$$\lim_{n\to\infty} (B_n,\phi_n) = (\lim_{n\to\infty} B_n, \lim_{n\to\infty} \phi_n) = (B,\phi).$$
Since local diffeomorphisms are closed, $\phi$ is a local diffeomorphism; if each $dB_n=0$ then, since the convergence is $C^\infty$, $dB=0$; thus $(B,\phi)$ is a local generalized diffeomorphism.  So the local generalized diffeomorphisms are closed.
\end{rem}

\begin{lem}
The action of local generalized diffeomorphisms on the deformations, $\Gamma(\^ ^2 L^*)$, of the standard generalized complex structure on $B_r \subset \C^n$, as in Definition \ref{generalized action on deformation}, is a left SCI-action.
\end{lem}
\begin{proof}
Since this action is defined over precisely the same $B_r$ as pushforward by local diffeomorphisms, we need only check the estimates \eqref{action estimate 1} and \eqref{action estimate 2}.

Let $\Phi=(B,\phi)$ be a local generalized diffeomorphism over a ball $B_r$.  Let $D\Phi$ be the ``derivative'' of of $\Phi$; that is, since $\Phi$ is a map of \emph{trivialized} Courant algebroids, $D\Phi$ is its fibrewise trivialization, a function from $B_r$ to the automorphisms of $T_0 B_r \dsum T^*_0 B_r \iso \R^{2n}$.  In terms of $B$ and $\phi$, at a point $x \in B_r$ $D\Phi$ acts as
$$\left(D\phi|_x \dsum (D\phi|_x^T)^{-1}\right) \;\comp\; e^{B|_x},$$
i.e., first by $B$-transform and then by derivative of $\phi$.  If $u \in \Gamma(TB_r \dsum T^*B_r)$ then, similarly to the proof of Lemma \ref{pushforward and pullback are SCI},
$$\Phi_* u = (D\Phi\cdot u) \comp\phi^{-1}.$$

Now we consider a deformation $\epsilon \in \Gamma(\^ ^2 L^*)$, regarding it as a map $L \to \bar{L}$.  A section of $L_{\Phi\cdot\epsilon}$ is uniquely represented as $u + (\Phi\cdot\epsilon)(u)$, for some $u \in \Gamma(L)$.  By definition, this is also the image of a section $v + \epsilon(v) \in \Gamma(L_\epsilon)$ under the pushforward $\Phi_*$, for some $v \in \Gamma(L)$.  Then
\begin{eqnarray*}
u + (\Phi\cdot\epsilon)(u) &=& (\Phi_*\comp(\Id+\epsilon))(v) \\
&=& (D\Phi\cdot(\Id+\epsilon)\cdot v) \comp \phi^{-1}
\end{eqnarray*}

We decompose the right hand side into $L$ and $\bar{L}$ components, and equate with the corresponding components on the left hand side.  First, the $L$ component:
$$u = \left( (D\Phi\cdot(\Id+\epsilon))^L_L \cdot v\right) \comp \phi^{-1}$$
Then
$$v = \left. (D\Phi\cdot(\Id+\epsilon))^L_L \right.^{-1} \cdot (u\comp\phi).$$

For the $\bar{L}$ component,
\begin{eqnarray}
(\Phi\cdot\epsilon)(u) &=& \left((D\Phi\cdot(\Id+\epsilon))^{\bar{L}}_L \cdot v \right) \comp \phi^{-1} \notag \\
&=& \left((D\Phi\cdot(\Id+\epsilon))^{\bar{L}}_L \cdot \left. (D\Phi\cdot(\Id+\epsilon))^L_L \right.^{-1} \cdot (u\comp\phi) \right) \comp \phi^{-1} \notag \\
&=& \left(\left((D\Phi\cdot(\Id+\epsilon))^{\bar{L}}_L \cdot \left. (D\Phi\cdot(\Id+\epsilon))^L_L \right.^{-1}\right) \comp \phi^{-1}\right) \cdot u \label{action formula}
\end{eqnarray}

If we drop the $u$ from either side, we have an explicit expression for $\Phi\cdot\epsilon$.  It is constructed from $\epsilon$, $\Phi$, $D\Phi$ and $\Id$ through the operations of sum, matrix multiplication and matrix inverse, pushforward by functions, and restriction and projection (to $L$ and $\bar{L}$).  Each of these operations is an SCI-action in the weak sense of Lemma \ref{composite action}, and so the result follows.
\end{proof}

\section{Checking the hypotheses of the abstract normal form theorem}\label{verifying hypotheses}

\subsection{Preliminary estimates}

\begin{lem}\label{bilinear estimate}
If $\Theta : V_1 \times V_2 \to W$ is a bilinear function between normed finite-dimensional vector spaces, and $f : U \to V_1$ and $g : U \to V_2$ are smooth on a compact domain $U$, then, applying $\Theta$ to $f$ and $g$ pointwise,
$$\|\Theta(f,g)\|_k \leq C (\|f\|_k \|g\|_0 + \|f\|_0 \|g\|_k) \leq C'\, \|f\|_k \|g\|_k,$$
and of course, $\|\Theta(f,g)\|_k \leq \LP(\|f\|_k,\|g\|_k)$.
\end{lem}
\begin{proof}
As remarked in Proposition \ref{interpolation inequality}, as a consequence of the existence of smoothing operators on spaces of smooth functions, the \emph{interpolation inequality} holds---for nonnegative integers $p \geq q \geq r$ and any function $f$ as above,
$$\|f\|_q^{p-r} \leq C \|f\|_r^{p-q} \|f\|_p^{q-r}.$$
From this inequality, the result follows by a standard argument (see \cite[Cor. II.2.2.3]{Hamilton}). 
\end{proof}

\begin{lem}\label{bracket estimate}
If $\alpha\in \Gamma(\^ ^i L^*)$ and $\beta\in \Gamma(\^ ^j L^*)$, then for $k\geq0$,
$$\|[\alpha,\beta]\|_k \,\leq\, C \left(\|\alpha\|_{k+1} \|\beta\|_1 + \|\alpha\|_1 \|\beta\|_{k+1}\right) \,\leq\, C' \|\alpha\|_{k+1} \|\beta\|_{k+1},$$
and of course, $\|[\alpha,\beta]\|_k \,\leq\, \LP(\|\alpha\|_{k+1},\, \|\beta\|_{k+1})$.
\end{lem}
\begin{proof}
If $\alpha$ and $\beta$ are generalized vector fields, there are pointwise-bilinear functions $\Theta$ and $\Lambda$ which express the Courant bracket formula \eqref{bracket formula} as
$$[\alpha,\beta] = \Theta(\alpha,\beta^{(1)}) - \Lambda(\beta,\alpha^{(1)}).$$
Then by Lemma \ref{bilinear estimate},
\begin{eqnarray*}
\|[\alpha,\beta]\|_k &\leq& C' (\|\alpha\|_k \|\beta^{(1)}\|_0 + \|\alpha\|_0 \|\beta^{(1)}\|_k \\
& & +\; \|\alpha^{(1)}\|_0 \|\beta\|_k + \|\alpha^{(1)}\|_k \|\beta\|_0 ) \\
&\leq& C \left(\|\alpha\|_{k+1} \|\beta\|_1 + \|\alpha\|_1 \|\beta\|_{k+1}\right)
\end{eqnarray*}
If $\alpha$ and $\beta$ are higher-rank tensors and the bracket is the generalized Schouten bracket, a suitable choice of $\Theta'$ and $\Lambda'$ will give the same result.
\end{proof}


\subsection{Verifying estimates \eqref{MMZ1}, \eqref{MMZ2}, \eqref{MMZ3} and \eqref{MMZ4}}

Recall that if $\epsilon = \epsilon_1 + \epsilon_2 + \epsilon_3 \in \Gamma(\^ ^2 L^*)$, with the terms being a bivector, a mixed term and 2-form respectively, then $\zeta(\epsilon) = \epsilon_2 + \epsilon_3$.  Then the following is an obvious consequence of our choice of norms.

\begin{lem}[Estimate \ref{MMZ1}]\label{projection estimate}
For all $\epsilon \in \Gamma(\^ ^2 L^*)$ and any $k$, $\|\zeta(\epsilon)\|_k \,\leq\, \|\epsilon\|_k$.
\end{lem}

We recall following estimate, taken from \cite{NijenhuisWoolf}, which was mentioned in Lemma \ref{existence of P}:
\begin{lem}\label{P estimate}
For all $\epsilon \in \Gamma(\^ ^2 L^*)$ and any $k$, $\|P\epsilon\|_k \,\leq\, C\,\|\epsilon\|_k$.
\end{lem}

\begin{lem}[Estimate \ref{MMZ2}]\label{V estimate}
For any $\epsilon \in \Gamma(\^ ^2 L^*)$ and large enough $k$,
$$\|V(\epsilon)\|_k \leq C\, \|\zeta(\epsilon)\|_{k+1}\, (1+\|\epsilon\|_{k+1}).$$
\end{lem}
\begin{proof}
$V(\epsilon) = P[\epsilon_1,P\epsilon_3] - P\zeta(\epsilon)$.  But $\|\epsilon_1\|_k \leq \|\epsilon\|_k$ and $\|\epsilon_3\|_k \leq \|\zeta(\epsilon)\|_k$, so by applying the triangle inequality and then Lemmas \ref{P estimate} and \ref{bracket estimate} the result follows.
\end{proof}

\begin{lem}[Estimate \ref{MMZ3}]\label{flow estimate}
For any $v \in \Gamma(L^*)$, any $0\leq t\leq1$ and large enough $k$,
$$\|\Phi_{tv} - Id\|_k \leq \LP(\|v\|_k)$$
\end{lem}
\begin{proof}
Let $v = X+\xi$, where $X$ is a vector field and $\xi$ is a 1-form, and let $\Phi_{tv} = (B_{tv},\phi_{tX})$.  From \cite{MirandaMonnierZung} we know that a counterpart of this lemma holds for the local diffeomorphism $\phi_{tX}$, therefore we are only concerned with $B_{tv}$. By the SCI-action estimate \eqref{action estimate 1} for pullbacks of differential forms,
\begin{equation}\label{e300}
\|\phi_{tX}^* d\xi\|_{k-1} \,\leq\, \LP(\|\xi\|_k,\, 1+\|\phi_{tX}-\Id\|_k)
\end{equation}
The counterpart of this Lemma in \cite{MirandaMonnierZung} tells us that
$$\|\phi_{tX}-\Id\|_k \leq \LP(\|X\|_k).$$
We plug this into \eqref{e300} and recall that $\|v\|_k = \sup(\|X\|_k,\,\|\xi\|_k)$; then,
$$\|\phi^*_{tX} d\xi\|_{k-1} \,\leq\, \LP(\|v\|_k).$$
But
\begin{eqnarray*}
\|B_{tv}\|_{k-1} &=& \left\| \int_0^t (\phi^*_{\tau X} d\xi)\, d\tau\right\|_{k-1} \\
&\leq& \int_0^t \left\|\phi^*_{\tau X} d\xi\right\|_{k-1}\, d\tau \\
&\leq& \int_0^t \LP(\|v\|_k)\, d\tau
\end{eqnarray*}
and the result follows.
\end{proof}

\begin{lem}[Estimate \ref{MMZ4}]
There is some $s$ such that, for any $v,w \in \Gamma(L^*)$, any integrable deformation $\epsilon \in \Gamma(\^ ^2 L^*)$, and large enough $k$,
\begin{eqnarray}
\|\Phi_v\cdot \epsilon - \Phi_w\cdot\epsilon\|_k &\leq& \LP(\|v - w\|_{k+1},\, 1 + \|v\|_{k+2} + \|w\|_{k+2}+\|\epsilon\|_{k+1}) \notag \\
& & \;+\; \LP\left((\|v\|_{k+3} + \|w\|_{k+3})^2,\, 1+\|\epsilon\|_{k+2}\right) \label{e302}
\end{eqnarray}
\end{lem}
\begin{proof}
The integral form of Proposition \ref{infinitesimal flow} tells us that
\begin{eqnarray*}
& &\Phi_v \cdot \epsilon - \Phi_w \cdot \epsilon \\
&=& \int_0^1 \left(\bar\del v + [v,\phi_{tv}\cdot\epsilon]\right)dt \,-\, \int_0^1 \left(\bar\del w + [w,\phi_{tw}\cdot\epsilon]\right)dt \notag \\
&=& \bar\del(v-w) \,+\, \int_0^1 [v-w\,,\;\phi_{tv}\cdot\epsilon]\,dt
\,+\, \int_0^1 [w\,,\;\phi_{tv}\cdot\epsilon-\phi_{tw}\cdot\epsilon]\,dt \\
\end{eqnarray*}
Integrating again, this time within the second Courant bracket, we get
\begin{eqnarray}
& & \bar\del(v-w) \;+\; \int_0^1 [v-w\,,\;\phi_{tv}\cdot\epsilon]\,dt \label{e301} \\
& & \quad +\; \int_0^1 \int_0^t \left[w\,,\; \bar\del(v-w)\, + [v,\phi_{\tau v}\cdot\epsilon] \,-\, [w,\phi_{\tau w}\cdot\epsilon]\right]\,d\tau\,dt \notag
\end{eqnarray}

To estimate $\|\Phi_v\cdot \epsilon - \Phi_w\cdot \epsilon\|_k$, we apply the triangle inequality to \eqref{e301}, and consider the three terms in turn.  Clearly, $\|\bar\del(v-w)\|_k$ is bounded by the first term in \eqref{e302}.

An aside: using the action estimate \eqref{nonlinear action estimate} and then Lemma \ref{flow estimate}, we see that
$$\|\phi_{tv}\cdot\epsilon\|_k \,\leq\, \LP(\|\epsilon\|_k + \|v\|_{k+1}).$$

Turning now to the second term of \eqref{e301}, we carry the norm inside the integral then, using the bracket estimate (Lemma \ref{bracket estimate}) and the above remark, we see that this term is bounded by the first term in \eqref{e302}.  Similarly, the third term in \eqref{e301} will be bounded by terms which have a factor of $\|v-w\|$, $\|w\|\cdot\|v\|$ or $\|w\|^2$.  Counting the total number of derivatives lost on each factor, the result follows.
\end{proof}

\subsection{Lemmas for estimate \eqref{MMZ5}}

The following lemma says that in our case the operator $\bar\del + [\epsilon_1,\cdot]$ is a good approximation of the deformed Lie algebroid differential $\bar\del + [\epsilon,\cdot]$.
\begin{lem}\label{estimate d_L}
For any $\epsilon \in \Gamma(\^ ^2 L^*)$ and large enough $k$,
$$\left\|\left(\bar\del V(\epsilon) + [\epsilon,V(\epsilon)]\right) - \left(\bar\del V(\epsilon) + [\epsilon_1,V(\epsilon)]\right)\right\|_k \,\leq\,
C \|\zeta(\epsilon)\|_{k+2}^2 \; (1 + \|\epsilon\|_{k+2})$$
\end{lem}
\begin{proof}
\begin{eqnarray*}
\left\|\left(\bar\del V(\epsilon) + [\epsilon,V(\epsilon)]\right) - \left(\bar\del V(\epsilon) + [\epsilon_1,V(\epsilon)]\right)\right\|_k &=& \|[\zeta(\epsilon),V(\epsilon)]\|_k \\
&\leq& C \|\zeta(\epsilon)\|_{k+1}\, \|V(\epsilon)]\|_{k+1} \\
&\leq& C \|\zeta(\epsilon)\|_{k+2}^2\, (1+\|\epsilon\|_{k+1}),
\end{eqnarray*}
(using Lemma \ref{V estimate} for the last step).
\end{proof}

The following lemma should be viewed as an approximate version of Proposition \ref{infinitesimal correction}, telling us that the infinitesimal action of $V(\epsilon)$ on $\epsilon$ almost eliminates the non-bivector component.
\begin{lem}\label{almost infinitesimal}
For an \emph{integrable} deformation $\epsilon \in \Gamma(\^ ^2 L^*)$ and large enough $k$,
$$\|\zeta\left(\bar\del V(\epsilon) + [\epsilon_1,V(\epsilon)] + \epsilon\right)\|_k \;\leq\; C \|\zeta(\epsilon)\|_{k+2}^2 \, (1 + \|\epsilon\|_{k+2})$$
\end{lem}
\begin{proof}
\begin{eqnarray*}
\bar\del V(\epsilon) + [\epsilon_1,V(\epsilon)] + \epsilon &=&
\bar\del P[\epsilon_1,P\epsilon_3] - \bar\del P\epsilon_2 - \bar\del P\epsilon_3 \\
& & +\; [\epsilon_1,P[\epsilon_1,P\epsilon_3]] - [\epsilon_1,P\epsilon_2] - [\epsilon_1,P\epsilon_3] + \epsilon
\end{eqnarray*}
The terms $[\epsilon_1,P\epsilon_2]$ and $[\epsilon_1,P[\epsilon_1,P\epsilon_3]]$ lie in $\^ ^2 T_{1,0}$, so when we project to the non-bivector part we get
$$\zeta\left(\bar\del V(\epsilon) + [\epsilon_1,V(\epsilon)] + \epsilon\right) \;=\;
\bar\del P[\epsilon_1,P\epsilon_3] - \bar\del P\epsilon_2 - \bar\del P\epsilon_3
- [\epsilon_1,P\epsilon_3] + \zeta(\epsilon),$$
(where $\zeta(\epsilon) = \epsilon_2 + \epsilon_3$).  This is the quantity we would like to bound.  We apply the identity $\bar\del P = 1 - P \bar\del$ \eqref{P error} to the first three terms on the right hand side, giving us
\begin{eqnarray}
& & [\epsilon_1,P\epsilon_3] - P\bar\del[\epsilon_1,P\epsilon_3] - \epsilon_2 + P\bar\del\epsilon_2 - \epsilon_3 + P\bar\del\epsilon_3 - [\epsilon_1,P\epsilon_3] + \epsilon \notag \\
&=& -P\bar\del[\epsilon_1,P\epsilon_3] + P\bar\del\epsilon_2 + P\bar\del\epsilon_3 \label{error 1}
\end{eqnarray}
We now use the fact that $\epsilon$ satisfies the Maurer-Cartan equations, \eqref{MC1} through \eqref{MC4}.  By \eqref{MC4}, $P\bar\del\epsilon_3 = -[\epsilon_2,\epsilon_3]$.  By \eqref{MC3},
\begin{equation}\label{term 1}
P\bar\del\epsilon_2 = -P\left( \frac{1}{2} [\epsilon_2,\epsilon_2] + [\epsilon_1,\epsilon_3] \right).
\end{equation}
By equation \eqref{anticommute error},
\begin{eqnarray}
-P\bar\del[\epsilon_1,P\epsilon_3] &=& P[\epsilon_1,\bar\del P\epsilon_3]
- P[\bar\del\epsilon_1, P\epsilon_3] \notag \\
&=& P[\epsilon_1,\epsilon_3] - P[\bar\del\epsilon_1, P\epsilon_3] \label{term 2}
\end{eqnarray}
$P[\epsilon_1,\epsilon_3]$ cancels between \eqref{term 1} and \eqref{term 2}.  Thus \eqref{error 1} becomes
$$-\frac{1}{2} P[\epsilon_2,\epsilon_2] - P[\bar\del\epsilon_1, P\epsilon_3] - P[\epsilon_2,\epsilon_3].$$
Applying \eqref{MC2} to $\bar\del\epsilon_1$, this is
$$-\frac{1}{2} P[\epsilon_2,\epsilon_2] + P[[\epsilon_1,\epsilon_2], P\epsilon_3] - P[\epsilon_2,\epsilon_3].$$
Through applications of Lemmas \ref{P estimate} and \ref{bracket estimate}, we find that this has $k$-norm bounded by
\begin{eqnarray*}
& &C \left( \|\epsilon_2\|_{k+1}^2 \,+\, \|\epsilon_1\|_{k+2} \, \|\epsilon_2\|_{k+2} \, \|\epsilon_3\|_{k+1}  
\,+\, \|\epsilon_2\|_{k+1}\, \|\epsilon_3\|_{k+1} \right) \\
&\leq& C'\, \|\zeta(\epsilon)\|_{k+2}^2 \, (1 + \|\epsilon\|_{k+2})
\end{eqnarray*}
The result follows.
\end{proof}

The following lemma is a version of Taylor's theorem.
\begin{lem}\label{Taylor estimate}
There is some $s$ such that for any integrable deformation $\epsilon \in \Gamma(\^ ^2 L^*)$, any $v \in \Gamma(L^*)$, and large enough $k$, 
$$\|(\Phi_v\cdot\epsilon - \epsilon) - (\bar\del v + [\epsilon,v])\|_k \,\leq\, \LP(1+\|\epsilon\|_{k+s},\, \|v\|_{k+s}^2).$$
\end{lem}
\begin{proof}
Applying the integral form of Proposition \ref{infinitesimal flow}, we see that
\begin{eqnarray*}
\left\|(\Phi_v\cdot\epsilon - \epsilon) - (\bar\del v + [\epsilon,v])\right\|_k
&=& \left\|\int_0^1 (\bar\del v + [\Phi_{tv}\cdot\epsilon,v])\, dt - (\bar\del v + [\epsilon,v])\right\|_k \\
&=& \left\|\int_0^1 [\Phi_{tv}\cdot\epsilon - \epsilon,\,v]\, dt\right\|_k \\
&\leq& \int_0^1 \LP(\|v\|_{k+1},\, \|\Phi_{tv}\cdot\epsilon - \epsilon\|_{k+1})\, dt
\end{eqnarray*}
Where in the last line we have carried the norm inside the integral and applied Lemma \eqref{bracket estimate}.  Applying the second axiom of SCI-actions \eqref{action estimate 2} and then Lemma \ref{flow estimate}, for some $s$ and $s'$,
\begin{eqnarray*}
\|\Phi_{tv}\cdot\epsilon - \epsilon\|_{k+1} &\leq& \LP(1+\|\epsilon\|_{k+s'},\, \|\Phi_{tv}-\Id\|_{k+s'}) \\
&\leq& \LP(1+\|\epsilon\|_{k+s},\|v\|_{k+s})
\end{eqnarray*}
Integrating the above estimate, the result follows.
\end{proof}

\begin{lem}[Estimate \ref{MMZ5}]
There is some $s$ such that, for any integrable deformation $\epsilon \in \Gamma(\^ ^2 L^*)$ and large enough $k$,
$$\|\zeta(\Phi_{V(\epsilon)} \cdot \epsilon)\|_k \leq \|\zeta(\epsilon)\|_{k+s}^{1+\delta} Poly(\|\epsilon\|_{k+s}, \|\Phi_{V(\epsilon)}-\Id\|_{k+s}, \|\zeta(\epsilon)\|_{k+s}, \|\epsilon\|_k),$$
where in this case the polynomial degree in $\|\epsilon\|_{k+s}$ does not depend on $k$.
\end{lem}
\begin{proof}
This is just an application of the triangle inequality using the estimates in this section.  We will show that, in the following series of approximations, terms on either side of a $\sim$ are close in the sense required:
$$\zeta(\Phi_{V(\epsilon)}\cdot\epsilon - \epsilon)
\;\sim\; \zeta(\bar\del V(\epsilon) + [\epsilon,V(\epsilon)])
\;\sim\; \zeta(\bar\del V(\epsilon) + [\epsilon_1,V(\epsilon)])
\;\sim\; -\zeta(\epsilon)$$
If so, then $\zeta(\Phi_{V(\epsilon)}\cdot\epsilon) \;\sim\; 0$ as required.

Applying the estimate for $V(\epsilon)$ (Lemma \ref{V estimate}) to Lemma \ref{Taylor estimate}, we see that
\begin{eqnarray*}
& & \|(\Phi_{V(\epsilon)}\cdot\epsilon - \epsilon) - (\bar\del V(\epsilon) + [\epsilon,V(\epsilon)])\|_k \\
& & \qquad\qquad   \leq\; \LP\left(1+\|\epsilon\|_{k+s},\, \|\zeta(\epsilon)\|_{k+s'+1}^2\,(1+\|\epsilon\|_{k+s'+1})\right).
\end{eqnarray*}
Applying $\zeta$ to the left hand side, this is the first approximation above.  (We remark that for large $k$, $\fl{(k+s)/2}+1 \leq k$, so we have a strictly limited degree in $\|\epsilon\|_l,\, l>k$.)  The remaining approximations are Lemma \ref{estimate d_L} (after applying $\zeta$ to its left hand side) and Lemma \ref{almost infinitesimal} respectively. 
\end{proof}

As remarked in Section \ref{prove main lemma}, we should now consider the Main Lemma proved.

\section{Main Lemma implies Main Theorem}\label{lemma implies theorem}

It is certainly the case that, near a complex point, a generalized complex structure is a deformation of a complex structure.  However, this deformation may not be small in the sense we need.  Therefore we use two means to control its size.

In this section, by $\delta_t : \C^n \to \C^n$ we will mean the dilation, $x \mapsto tx$.  If $\epsilon$ is a tensor on $\C^n$, then by $\delta_t \epsilon$ we mean the \emph{pushforward} of $\epsilon$ under the dilation map $x \mapsto tx$.  The complex structure on $\C^n$ is invariant under $\delta_t$; therefore if $\epsilon \in \Gamma(\^ ^2 L^*)$ is a deformation of the complex structure, then $\delta_t \epsilon = (0,\delta_t)\cdot\epsilon$ as in Definition \ref{generalized action on deformation}.

Suppose that $\epsilon \in \Gamma(\^ ^2 L^*)$, where $L^* = T_{1,0} \dsum T^*_{0,1}$, and that $\epsilon$ is decomposed into $\epsilon_1 + \epsilon_2 + \epsilon_3$, where $\epsilon_1$ is a bivector, $\epsilon_3$ is a 2-form and $\epsilon_2$ is of mixed type, as in Section \ref{Maurer-Cartan section}.  We wish to see how $\delta_t$ acts on these terms.

\begin{prop}\label{delta action}
Suppose that $t>0$.  For any $x \in \C^n$ and any $k$ we have the following pointwise norm comparisons for derivatives of $\epsilon$, before and after the dilation.  Let $k\geq0$.  Then
\begin{eqnarray*}
\|(\delta_t\epsilon_1)^{(k)}(tx)\|_0 &\leq& t^{2-k}\;\|\epsilon_1^{(k)}(x)\|_0 \\[3pt]
\quad \|(\delta_t\epsilon_2)^{(k)}(tx)\|_0 &\leq& \;t^{-k}\;\|\epsilon_2^{(k)}(x)\|_0 \\[3pt]
\textnormal{and}\quad \|(\delta_t\epsilon_3)^{(k)}(tx)\,\|_0 &\leq& t^{-2-k}\|\epsilon_3^{(k)}(x)\|_0.
\end{eqnarray*}
\end{prop}

\begin{proof}
Under a dilation, vectors scale with $t$ and covectors scale inversely with $t$.  Then
\begin{equation}\label{scaling 1}
(\delta_t\epsilon_1)(tx) = t^2\epsilon_1(x), \quad (\delta_t\epsilon_2)(tx) = \epsilon_2(x)
\quad\textnormal{and}\quad (\delta_t\epsilon_3)(tx) = t^{-2}\epsilon_3(x).
\end{equation}
If $x_i$ is a coordinate and $f$ a tensor, then
$$\frac{\del}{\del x_i}(\delta_t f)(tx) = \delta_t\left(\frac{\del}{\del tx_i} f\right)(tx)
= t^{-1}\delta_t\left(\frac{\del}{\del x_i} f\right)(tx)$$
This tells us that
$$\left\|(\delta_t f)^{(k+1)}(tx)\right\|_0
\leq t^{-1}\left\|\delta_t\left(f^{(k)}\right)(tx) \right\|_0$$
By induction on this inequality and then applying the formulas in \eqref{scaling 1}, the result follows.
\end{proof}

We now define the $\lambda$-transform, which is not a Courant isomorphism, but which does take generalized complex structures to generalized complex structures.

\begin{defn}
If $t>0$, let $\lambda_t : T \dsum T^* \to T \dsum T^*$ so that $\lambda_t(X,\xi) = (tX,\xi)$.  Then $\lambda_t$ also acts on generalized complex structures by mapping their eigenbundles (or by conjugating $J$).
\end{defn}

$\lambda_t$ commutes with diffeomorphisms, but it does not quite commute with Courant isomorphisms.
\begin{notn}
If $\Phi = (B,\phi)$ is a Courant isomorphism, then let $\lambda_t\cdot\Phi = (t^{-1}B,\phi)$.
\end{notn}

\begin{prop}
If $\Phi$ is a Courant isomorphism then
$$\Phi \comp \lambda_t = \lambda_t \comp (\lambda_t\cdot\Phi).$$
\end{prop}

Again we consider a deformation $\epsilon = \epsilon_1 + \epsilon_2 + \epsilon_3$ of the complex structure on $\C^n$.  $\lambda_t(L_\epsilon)$ will be another generalized complex structure.

\begin{prop}\label{lambda action}
$\lambda_t(L_\epsilon) = L_{\lambda_t\epsilon}$, where
$$\lambda_t\epsilon = t \epsilon_1 + \epsilon_2 + t^{-1} \epsilon_3.$$
\end{prop}

\begin{rem}
We can check that this transformation respects the Maurer-Cartan equations, \eqref{MC1} through \eqref{MC4}, which tells us that if $L_\epsilon$ was generalized complex then so is $\lambda_t(L_\epsilon)$.
\end{rem}

We can now prove that the Main Theorem follows from the Main Lemma.  Recall: 
\begin{main lem}
Let $J$ be a generalized complex structure on the closed unit ball $B_1$ about the origin in $\C^n$.  Suppose that $J$ is a small enough deformation of the complex structure on $B_1$, and suppose that $J$ is of complex type at the origin.  Then, in a neighbourhood of the origin, $J$ is equivalent to a deformation of the complex structure by a holomorphic Poisson structure on $\C^n$.
\end{main lem}

\begin{main thm}
Let $J$ be a generalized complex structure on a manifold $M$ which is of complex type at point $p$.  Then, in a neighbourhood of $p$, $J$ is equivalent to a generalized complex structure induced by a holomorphic Poisson structure, for some complex structure near $p$.
\end{main thm}

\begin{proof}[Proof of Main Theorem from Main Lemma]
Suppose that $J$ is a generalized complex structure on $M$, with $p$ a point of complex type.  We may assume without loss of generality that $p=0$ in the closed unit ball $B_1 \subset \C^n$, where the complex structure on $T_0 \C^n$ induced by $J$ agrees with the standard one.  By application of an appropriate $B$-transform, we may assume that, at $0$, $J$ agrees with the standard generalized complex structure, $J_{\C^n}$, for $\C^n$.  Then, near $0$, $J$ is a deformation of $J_{\C^n}$ by $\epsilon \in \Gamma(\^ ^2 L^*)$, and $\epsilon$ vanishes at $0$.

For $t>0$, let
$$R_t \epsilon = \delta_{t^{-1}}\, \lambda_{t^2}\, \epsilon = \lambda_{t^2}\, \delta_{t^{-1}}\, \epsilon.$$
Since $\epsilon$ (and hence $R_t \epsilon$) vanishes at $0$ to at least first order, there is some $C>0$ such that, for all $0 < t \leq 1$,
\begin{equation}\label{e701}
\|(R_t \epsilon) (x)\|_0 \leq C\,t\, \|(R_t \epsilon)(t^{-1}x)\|_0.
\end{equation}
For derivatives $k>0$, we apply Proposition \ref{delta action} to the components $\epsilon_1$, $\epsilon_2$ and $\epsilon_3$, and
\begin{eqnarray*}
\|(R_t \epsilon_1)^{(k)}(t^{-1}x)\|_0 &\leq& t^{-2+k}\, \|(\lambda_{t^2}\epsilon_1)^{(k)}(x)\|_k, \\
\|(R_t \epsilon_2)^{(k)}(t^{-1}x)\|_0 &\leq& \; t^k\quad \|(\lambda_{t^2}\epsilon_2)^{(k)}(x)\|_k \quad\textnormal{and} \\
\|(R_t \epsilon_3)^{(k)}(t^{-1}x)\|_0 &\leq& t^{2+k}\; \|(\lambda_{t^2}\epsilon_3)^{(k)}(x)\|_k.
\end{eqnarray*}
Using Proposition \ref{lambda action}, we apply $\lambda_{t^2}$, so that
\begin{eqnarray*}
\|(R_t \epsilon_1)^{(k)}(t^{-1}x)\|_0 &\leq& t^k\, \|\epsilon_1^{(k)}(x)\|_k, \\
\|(R_t \epsilon_2)^{(k)}(t^{-1}x)\|_0 &\leq& \; t^k\, \|\epsilon_2^{(k)}(x)\|_k \quad\textnormal{and} \\
\|(R_t \epsilon_3)^{(k)}(t^{-1}x)\|_0 &\leq& t^k\, \|\epsilon_3^{(k)}(x)\|_k.
\end{eqnarray*}
Thus, if $k>0$ or (because of \eqref{e701}, if $k=0$ also), we have
$$\|(R_t \epsilon)^{(k)} (x)\|_0 \leq C\,t\, \|\epsilon^{(k)}(x)\|_0.$$
So,
$$\|(R_t \epsilon) (x)\|_k \leq C\,t\, \|\epsilon(x)\|_k.$$

Taking the $\sup$-norm always over the fixed set $B_1$, we have that $\|R_t \epsilon\|_k \leq C\,t\, \|\epsilon\|_k$.
Thus $\|R_t \epsilon\|_k$ is as small as we like for some $t$, and satisfies the hypotheses of the Main Lemma; so there exists a local Courant isomorphism $\Phi$ such that $\Phi_t \cdot R_t \epsilon = \beta$, where $\beta$ is a holomorphic Poisson bivector.  Then
\begin{eqnarray*}
\beta &=& \Phi_t \cdot \left(\lambda_{t^2}\, \delta_{t^{-1}}\, \epsilon\right) \\
&=& \lambda_{t^2} \left((\lambda_{t^2}\cdot\Phi_t) \cdot \delta_{t^{-1}}\, \epsilon\right)
\end{eqnarray*}
But the action of $\lambda_{t^2}$ on a bivector is just scaling by $t^2$, so
$$(\lambda_{t^2}\cdot\Phi_t) \cdot \delta_{t^{-1}}\, \epsilon = t^{-2}\beta.$$
Thus, starting from a suitably small neighbourhood of $0$, by applying first the dilation $\delta_{t^{-1}}$ and then the Courant isomorphism $\lambda_{t^2}\cdot\Phi_t$, we see that $\epsilon$ is locally equivalent to the holomorphic Poisson structure $t^{-2}\beta$.
\end{proof}

\end{document}